\documentclass[11pt]{article}
\usepackage[margin=1.1in]{geometry}
\usepackage{amsmath,amssymb,amsthm,mathtools}
\usepackage{enumitem}
\usepackage{booktabs,tabularx}
\newcolumntype{Y}{>{\raggedright\arraybackslash}X}
\usepackage{cite}
\usepackage[colorlinks=true,linkcolor=blue,citecolor=blue,urlcolor=blue,
  pdftitle={Cyclic Sources of Strong Domination in Graph Norms},
  pdfauthor={Shuyan Chen}]{hyperref}
\usepackage{microtype}
\usepackage{tikz}
\usetikzlibrary{arrows.meta,calc,positioning}

\newtheorem{theorem}{Theorem}[section]
\newtheorem{corollary}[theorem]{Corollary}
\newtheorem{proposition}[theorem]{Proposition}
\newtheorem{lemma}[theorem]{Lemma}
\newtheorem{claim}[theorem]{Claim}
\newtheorem{problem}[theorem]{Problem}
\newtheorem{conjecture}[theorem]{Conjecture}
\theoremstyle{definition}
\newtheorem{definition}[theorem]{Definition}
\newtheorem{remark}[theorem]{Remark}

\newcommand{\E}{\mathbb E}

\newcommand{\one}{\mathbf 1}

\newcommand{\diam}{\operatorname{diam}}
\newcommand{\tr}{\operatorname{tr}}

\newcommand{\preceqG}{\preceq_{\mathrm G}}
\newcommand{\preceqop}{\preceq_{\mathrm{op}}}
\newcommand{\succeqop}{\succeq_{\mathrm{op}}}

\title{Cyclic Sources of Strong Domination in Graph Norms}
\author{Shuyan Chen\\
\small Department of Mathematics, University of Manchester\\
\small \texttt{shuyan.chen-2@student.manchester.ac.uk}}
\date{}

\begin{document}
\maketitle

\begin{abstract}
Conlon and Lee asked for strongly dominating graphs beyond norming graphs and even paths.  We construct a two-parameter family of pairwise non-isomorphic $2$-connected strongly dominating graphs that are not seminorming, and hence lie outside the two classes of examples previously identified for signed strong domination.  The construction uses cyclic amalgamation of two-rooted blocks.  For root-reversible blocks, we characterize the generation of all even cyclic amalgams by local even-Schatten inequalities for transfer operators.  We determine this criterion for $K_{2,m}$, with the roots in the part of size $m$: it holds exactly when $m$ is even.  We also classify the connected outerplanar strongly dominating graphs and the connected root-reversible outerplanar blocks satisfying the universal cyclic criterion.
\end{abstract}

\section{Introduction}

For a finite graph $H$ and a bounded symmetric measurable kernel $W:[0,1]^2\to\mathbb R$, write
\[
 t_H(W)=\int_{[0,1]^{V(H)}}\prod_{uv\in E(H)}W(x_u,x_v)\prod_{v\in V(H)}dx_v,
 \qquad
 p_H(W)=|t_H(W)|^{1/e(H)}.
\]
In every domination statement, $F\subseteq H$ means a spanning edge-subgraph; allowing vertex deletion is equivalent because isolated vertices contribute a factor $1$.  Following Conlon and Lee, $H$ \emph{strongly dominates} $F$ if $p_H(W)\ge p_F(W)$ for every bounded symmetric real kernel $W$, and $H$ is \emph{strongly dominating} if this holds for every non-empty $F\subseteq H$ \cite[Section~6]{ConlonLeeDomination}.  Signed kernels permit cancellation, so this is substantially stronger than non-negative domination.  Conlon and Lee proved that seminorming graphs are strongly dominating, identified even paths as the only further examples then known, and asked for more \cite[Section~6]{ConlonLeeDomination}.

Strong domination should be distinguished from positivity and the graph-norm classes; the terminology used here is fixed in Remark~\ref{rem:graph-norm-terminology}.  Positive graphs require $t_H(W)\ge0$ for every real kernel and are conjecturally governed by gluing along an independent set \cite{PositiveGraphs}.  Norming, seminorming and weakly norming graphs have their own structural theory \cite{ConlonLeeReflection,GarbeHladkyLee,LeeSidorenkoComplex}.  The mechanism below is different: a local signed inequality is propagated through a two-vertex boundary by the noncommutative H\"older inequality.

Our main application is a two-parameter family of $2$-connected strongly dominating graphs outside the seminorming class.  We obtain it from a cyclic local-to-global transfer criterion, determine that criterion for two-rooted $K_{2,m}$ blocks, and classify its outerplanar sources.

Let $(B,\ell,r)$ be a graph with two distinguished non-adjacent roots.  Its cyclic amalgam $C_q(B)$ is obtained from $q$ copies of $B$ by identifying the right root of each copy with the left root of the next; Figure~\ref{fig:cyclic-amalgamation} shows the construction.  Informally, we call $B$ a \emph{cyclic source} when these repetitions around cycles produce strongly dominating graphs; the precise analytic condition is stated below.  The copies share root variables, so a subgraph density is a cyclic integral rather than a product of block densities.  For a birooted subgraph $S\subseteq B$, let $\mathsf T_S^W$ be the integral operator whose kernel is the conditional density of $S$ with the two root variables fixed.  The key identity is
\[
        t_F(W)=\operatorname{tr}\bigl(\mathsf T_{S_1}^W\cdots\mathsf T_{S_q}^W\bigr)
\]
for every subgraph $F\subseteq C_q(B)$, where $S_i$ is the restriction of $F$ to the $i$th block.  The global composition step is therefore the Schatten--von Neumann H\"older inequality.

\begin{figure}[t]
\centering
\begin{tikzpicture}[scale=0.9, every node/.style={font=\small}]
  \begin{scope}[xshift=-4.6cm]
    \node (l) at (-1.8,1.25) {$\ell$};
    \node (z1) at (-0.9,1.25) {$z_1$};
    \node (zd) at (0,1.25) {$\cdots$};
    \node (zm) at (0.9,1.25) {$z_{m-2}$};
    \node (r) at (1.8,1.25) {$r$};
    \node[circle,draw,inner sep=2pt] (a) at (-0.55,-0.15) {$a$};
    \node[circle,draw,inner sep=2pt] (b) at (0.55,-0.15) {$b$};
    \foreach \v in {l,z1,zm,r}{
      \draw (a)--(\v);
      \draw (b)--(\v);
    }
    \draw[densely dotted] (a)--(zd);
    \draw[densely dotted] (b)--(zd);
    \draw[rounded corners] (-2.2,1.6) rectangle (2.2,-0.55);
    \node at (0,-0.95) {(a) the block $\mathcal B_m=K_{2,m}$};
    \node[font=\scriptsize] at (0,1.9) {the roots lie in the size-$m$ part};
  \end{scope}

  \begin{scope}[xshift=3.4cm]
    \node[circle,draw,inner sep=2pt] (x1) at (0,1.65) {$x_1$};
    \node[circle,draw,inner sep=2pt] (x2) at (1.8,0) {$x_2$};
    \node[circle,draw,inner sep=2pt] (x3) at (0,-1.65) {$x_3$};
    \node[circle,draw,inner sep=2pt] (x4) at (-1.8,0) {$x_4$};
    \node[draw,rounded corners,fill=white,inner sep=3pt] at (1.0,0.95) {$B_1$};
    \node[draw,rounded corners,fill=white,inner sep=3pt] at (1.0,-0.95) {$B_2$};
    \node[draw,rounded corners,fill=white,inner sep=3pt] at (-1.0,-0.95) {$B_3$};
    \node[draw,rounded corners,fill=white,inner sep=3pt] at (-1.0,0.95) {$B_4$};
    \draw[thick] (x1) to[bend left=10] (x2);
    \draw[thick] (x2) to[bend left=10] (x3);
    \draw[thick] (x3) to[bend left=10] (x4);
    \draw[thick] (x4) to[bend left=10] (x1);
    \node at (0,-2.15) {(b) the cyclic amalgam $C_4(B)$};
    \node[font=\scriptsize,align=center] at (0,2.15) {$r_i$ is identified with $\ell_{i+1}$\\(indices modulo $4$)};
  \end{scope}
\end{tikzpicture}
\caption{A two-rooted block and its cyclic amalgamation.  The right-hand panel is schematic: each thick arc is replaced by a full copy of $B$.  }
\label{fig:cyclic-amalgamation}
\end{figure}
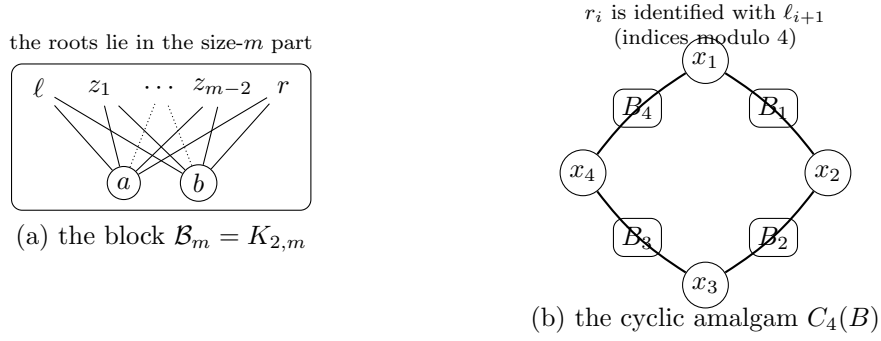

Call $B$ \emph{root-reversible} if an automorphism interchanges its roots.  We call it a \emph{universal even-Schatten block} if, for every even integer $p\ge2$, every bounded symmetric real kernel $W$, and every non-empty birooted subgraph $S\subseteq B$,
\begin{equation}\label{eq:intro-universal-schatten}
        \|\mathsf T_S^W\|_{S_p}^{1/e(S)}
        \le
        \|\mathsf T_B^W\|_{S_p}^{1/e(B)}.
\end{equation}
Root-reversibility is not part of the analytic definition; it is the symmetry needed for the converse.  Root non-adjacency is imposed whenever $C_2(B)$ is used.  The first theorem identifies the local condition with strong domination of every even cyclic amalgam.

\begin{theorem}[Universal cyclic transfer criterion]\label{thm:exact-cyclic-source}
Let $(B,\ell,r)$ be a root-reversible birooted graph with non-adjacent roots and $b=e(B)\ge1$.  The following are equivalent.
\begin{enumerate}[label=\textup{(\roman*)}]
\item $B$ is a universal even-Schatten block.
\item $C_q(B)$ is strongly dominating for every even integer $q\ge2$.
\end{enumerate}
If the local inequality also holds at every odd integer exponent and $\mathsf T_B^W$ is positive semidefinite for every $W$, then $C_q(B)$ is strongly dominating for every integer $q\ge2$.
\end{theorem}

At the exponent $p=2m$, the local condition in Theorem~\ref{thm:exact-cyclic-source} may equivalently be written as the spectral-moment inequality
\begin{equation}\label{eq:intro-moment-formulation}
 \bigl(\operatorname{tr}|\mathsf T_S^W|^{2m}\bigr)^{b}
 \le
 \bigl(\operatorname{tr}|\mathsf T_B^W|^{2m}\bigr)^{e(S)}
\end{equation}
for every non-empty $S\subseteq B$, every bounded symmetric real kernel $W$, and every integer $m\ge1$.  This moment form is convenient for certificates; the substantive equivalence is local-to-global.  A finite classification first appears for complete bipartite blocks.  For $m\ge2$, let $\mathcal B_m$ be $K_{2,m}$ rooted at two vertices in the part of size $m$, and put
\[
        N_{t,q}=C_q(\mathcal B_{2t}).
\]

\begin{theorem}[Complete-bipartite parity theorem]\label{thm:necklace-family}
For every $m,q\ge2$,
\[
        C_q(\mathcal B_m)\text{ is strongly dominating}
        \quad\Longleftrightarrow\quad
        m\text{ is even}.
\]
For $t,q\ge2$, the graphs $N_{t,q}$ are pairwise non-isomorphic and $2$-connected; none is seminorming (and hence none is norming), and none is a path.  At the boundary $t=1$, the graphs $C_q(\mathcal B_2)$ are norming reflection graphs.
\end{theorem}

These graphs lie outside both classes previously identified for signed strong domination: they are not seminorming, and their $2$-connected cyclic structure rules out even paths.  For even $m$ the proof gives a Gram factorization and local inequalities at every integer exponent $p\ge2$; for odd $m$ a rank-one signed kernel annihilates the full transfer but not a proper subgraph.

The spectral source condition also admits a finite combinatorial description in the outerplanar class.  Let $\mathcal D_k$ be the chain of $k$ copies of $C_4$ joined successively at opposite vertices, with the two unglued opposite vertices as roots.

\begin{theorem}[Outerplanar cyclic-source classification]\label{thm:outerplanar-source}
Let $(B,\ell,r)$ be a connected root-reversible outerplanar birooted graph with non-adjacent roots.  Then $B$ is a universal even-Schatten block if and only if either
\[
        B\cong P_k\quad(k\ge2),
        \qquad\text{with the roots at its endpoints},
\]
or
\[
        B\cong\mathcal D_k\quad(k\ge1).
\]
In both cases every even cyclic amalgam $C_q(B)$ is norming.  For $B=\mathcal D_k$, every $C_q(B)$ is norming without a parity restriction on $q$.
\end{theorem}

Thus outerplanar sources yield only norming cyclic families.  The blocks $\mathcal B_{2t}$ with $t\ge2$ cross the outerplanar boundary through a $K_{2,3}$ minor, while their cyclic closures acquire a $K_4$ minor.  The same signed obstructions also give a graph-level outerplanar classification.

\begin{theorem}[Outerplanar classification]\label{thm:outerplanar}
Let $H$ be a connected outerplanar graph with at least one edge.  Then $H$ is strongly dominating if and only if
\[
        H\cong K_2,
        \qquad H\cong P_{2s},
        \qquad H\cong K_{1,2s},
        \qquad\text{or}\qquad H\cong C_{2s},
\]
with $s\ge1$ in the path and star cases and $s\ge2$ in the cycle case.
\end{theorem}

A two-atom additive kernel forces
\[
        \tau(H)=\frac{e(H)}{\Delta(H)},
\]
so the maximum-degree vertices of a strongly dominating graph form a vertex cover.  Together with parity witnesses and degree-$2$ ears, this proves both outerplanar classifications and shows that every failure in the class has a two-point witness.  Appendix~\ref{app:one-root} records the scalar one-root analogue and its eventual no-go result; it is not used in the cyclic theory.

The proof architecture is as follows.  Section~2 fixes terminology and the known boundary examples.  Section~3 proves the cyclic transfer criterion.  Section~4 treats $K_{2,m}$: the even case is split into endpoint, interpolation, diagonal and off-diagonal estimates, while the odd case has a rank-one obstruction.  Section~5 develops the signed structural witnesses and the two outerplanar classifications.  Section~6 gives open problems, and Appendix~\ref{app:one-root} contains the one-root comparison.

\section{Preliminaries}

Throughout, $P_m$ denotes the path with $m$ edges and $C_m$ the cycle with $m$ edges.  We use \emph{series-parallel} in the minor-theoretic sense: a graph is series-parallel if and only if it is $K_4$-minor-free.

All graphs are finite and simple unless explicitly stated otherwise.  In domination inequalities, subgraphs are spanning edge-subgraphs; isolated vertices play no role in homomorphism densities, so this convention loses no generality.  Kernels are identified up to null sets, and finite weighted step kernels represent finite probability-space models; see \cite{BorgsChayesLovaszSosVesztergombi,LovaszBook}.

\begin{definition}
Let $H$ be a graph with $e(H)\ge1$.  For a bounded symmetric real kernel $W$, define
\[
        p_H(W)=|t_H(W)|^{1/e(H)}.
\]
For graphs $H,F$ with at least one edge, write $H\succeq_sF$ if $p_H(W)\ge p_F(W)$ for every bounded symmetric real kernel $W$.  This is Conlon--Lee's \emph{strong domination} relation \cite[Section~6]{ConlonLeeDomination}.  A graph $H$ is \emph{strongly dominating} if $H\succeq_sF$ for every non-empty subgraph $F\subseteq H$.
\end{definition}

The definitions are invariant under scaling: $p_H(cW)=|c|p_H(W)$.  Thus every parameter-dependent witness may be rescaled into $[-1,1]$.

\subsection{Known positive mechanisms}

For the graph-norm notions in this subsection, we use the standard bipartite convention: after fixing a bipartition, each edge is oriented from the first class to the second, and $W$ ranges over bounded real kernels which need not be symmetric.  This is stronger than the restriction to symmetric kernels used in strong domination.  For edge-decorated kernels $(W_e)_{e\in E(H)}$, write $t_H((W_e))$ for the corresponding decorated density.  A graph $H$ is \emph{seminorming} if $W\mapsto |t_H(W)|^{1/e(H)}$ is a seminorm, \emph{norming} if this functional is a norm, and \emph{weakly norming} if $W\mapsto t_H(|W|)^{1/e(H)}$ is a norm.

\begin{remark}[Graph-norm terminology]\label{rem:graph-norm-terminology}
Conlon and Lee \cite{ConlonLeeReflection} used the word ``norming'' for what is called \emph{seminorming} here.  Every invocation of their reflection-graph theorem is translated into the present terminology; a separate non-degeneracy argument is supplied whenever the stronger conclusion ``norming'' is asserted.  Thus
\[
        \text{norming}\Longrightarrow\text{seminorming}
        \Longrightarrow\text{weakly norming}.
\]
The second implication is \cite[Observation~2.5(i)]{Hatami}.  For the component structure of these classes see \cite{GarbeHladkyLee}; the real- and complex-valued notions of norming coincide by \cite{LeeSidorenkoComplex}.
\end{remark}

Hatami's decorated H\"older characterization states that $H$ is seminorming exactly when
\begin{equation}\label{eq:decorated-holder}
        |t_H((W_e)_{e\in E(H)})|
        \le
        \prod_{e\in E(H)}p_H(W_e)
\end{equation}
for every signed decoration \cite[Theorem~2.8]{Hatami}.

\begin{proposition}[Seminorming certificate]\label{prop:seminorming-certificate}
Every seminorming graph is strongly dominating.
\end{proposition}

\begin{proof}
For a non-empty subgraph $F\subseteq H$, put $W_e=W$ on $E(F)$ and $W_e=1$ elsewhere in \eqref{eq:decorated-holder}.  Then $t_H((W_e))=t_F(W)$ and $p_H(1)=1$, so $|t_F(W)|\le p_H(W)^{e(F)}$.
\end{proof}

Weakly norming is not enough for signed domination: $K_{1,3}$ is weakly norming but is killed by the odd-degree rank-one obstruction below.  The converse separation is supplied by even paths.

The non-negative path-moment inequality has the classical Blakley--Roy inequality as an antecedent \cite{BlakleyRoy}.

\begin{proposition}[Even paths]\label{prop:evenpaths}
For every $s\ge1$, the path $P_{2s}$ is strongly dominating.
\end{proposition}

\begin{proof}
Let $A=A_W$ be the self-adjoint integral operator induced by $W$, and let $\nu$ be its probability spectral measure at $\one$.  Then
\[
        t_{P_j}(W)=\int\lambda^j\,d\nu(\lambda),
        \qquad
        t_{P_{2s}}(W)=\|A^s\one\|_2^2.
\]
For $0\le j\le2s$, Lyapunov's inequality gives
\[
        |t_{P_j}(W)|
        \le
        \left(\int|\lambda|^{2s}\,d\nu\right)^{j/(2s)}
        =t_{P_{2s}}(W)^{j/(2s)}.
\]
A subgraph of a path is a disjoint union of paths, and densities multiply over components.
\end{proof}

\begin{proposition}[Even stars]\label{prop:evenstars}
For every $s\ge1$, the star $K_{1,2s}$ is strongly dominating.
\end{proposition}

\begin{proof}
With $g(x)=\int W(x,y)\,dy$, one has $t_{K_{1,j}}(W)=\int g^j$.  For $0\le j\le2s$,
\[
        \left|\int g^j\right|
        \le \|g\|_{L^j}^{j}
        \le \|g\|_{L^{2s}}^{j}
        =t_{K_{1,2s}}(W)^{j/(2s)}.
\]
A subgraph of a star is a smaller star together with isolated vertices, so this proves the claim.  This is also the classical H\"older certificate underlying the even-star seminorm.
\end{proof}

\begin{proposition}[Even cycles]\label{prop:evencycles}
For every $s\ge2$, the cycle $C_{2s}$ is strongly dominating.
\end{proposition}

\begin{proof}
Let $A=A_W$.  Since $A$ is self-adjoint,
\[
        t_{C_{2s}}(W)=\operatorname{tr}(A^{2s})=\|A\|_{S_{2s}}^{2s}.
\]
For every $0\le j\le 2s$,
\[
        |t_{P_j}(W)|
        =|\langle\one,A^j\one\rangle|
        \le \|A\|_{\mathrm{op}}^j
        \le \|A\|_{S_{2s}}^j
        =t_{C_{2s}}(W)^{j/(2s)}.
\]
Every proper subgraph $F\subset C_{2s}$ is a disjoint union of paths, say with edge lengths $j_1,\ldots,j_k$.  Multiplicativity over components therefore gives
\[
        |t_F(W)|
        \le t_{C_{2s}}(W)^{(j_1+\cdots+j_k)/(2s)}
        =t_{C_{2s}}(W)^{e(F)/(2s)}.
\]
This is the Schatten-norm example in graph norm theory \cite{Hatami,ConlonLeeReflection,SimonTraceIdeals}.
\end{proof}

Conlon and Lee record norming graphs and even paths as the previously known strongly dominating mechanisms \cite[Section~6]{ConlonLeeDomination}.  Even paths of length at least four lie outside the weakly norming class because weakly norming graphs are edge-transitive \cite[Theorem~1]{SidorenkoEdgeTransitive}.

\section{Cyclic amalgamation and spectral sources}\label{subsec:cyclic-amalgamation}

A \emph{birooted graph} is a triple $(B,\ell,r)$ with two distinguished vertices, called the left and right roots.  A birooted subgraph retains both roots as isolated vertices when necessary.  For a bounded symmetric kernel $W$, let
\[
        F_S^W(x,y)
\]
be the density of a birooted subgraph $S$ with its left and right roots fixed at $x$ and $y$, and let $\mathsf T_S^W$ be the integral operator on $L^2([0,1])$ with kernel $F_S^W$.  Since $W$ is bounded and $S$ is finite, $F_S^W$ is bounded; hence every $\mathsf T_S^W$ is Hilbert--Schmidt.  A product of at least two such operators is trace class, so all cyclic traces below are well defined.  If $S^{\mathrm{rev}}$ is obtained from $S$ by interchanging its roots, then symmetry of $W$ gives
\[
        F_{S^{\mathrm{rev}}}^W(x,y)=F_S^W(y,x),
        \qquad
        \mathsf T_{S^{\mathrm{rev}}}^W=(\mathsf T_S^W)^*.
\]

For the empty birooted subgraph, $F_\varnothing^W\equiv1$ and $\mathsf T_\varnothing^W$ is the rank-one averaging projection, whose Schatten norm is $1$ at every exponent.  Thus every local inequality with $S=\varnothing$ is automatic.  We state the substantive hypotheses for non-empty $S$, while allowing empty restrictions when a global subgraph is decomposed block by block.

\begin{lemma}[Cyclic trace formula for kernel operators]\label{lem:cyclic-trace-formula}
Let $q\ge2$, and let $A_1,\ldots,A_q$ be Hilbert--Schmidt integral operators on $L^2(\Omega)$, where $\Omega$ is a finite measure space, with kernels $a_1,\ldots,a_q\in L^2(\Omega^2)$.  Then $A_1\cdots A_q$ is trace class and
\begin{equation}\label{eq:abstract-cyclic-trace}
 \operatorname{tr}(A_1\cdots A_q)
 =\int_{\Omega^q}\prod_{i=1}^q a_i(x_i,x_{i+1})\,d\mathbf x,
 \qquad x_{q+1}=x_1.
\end{equation}
\end{lemma}

\begin{proof}
For finite-rank step kernels, \eqref{eq:abstract-cyclic-trace} is the usual matrix trace identity.  Approximate each $a_i$ in $L^2(\Omega^2)$ by such kernels; the corresponding operators converge in $S_2$.  The product is trace class because the product of two $S_2$ operators lies in $S_1$ and every remaining factor is bounded, with $\|A\|_{\mathrm{op}}\le\|A\|_{S_2}$.  A telescoping expansion, together with the Schatten H\"older inequality using exponents $2,2,\infty,\ldots,\infty$, shows that the traces of the approximating products converge.  The cycle integrals converge by the corresponding repeated Cauchy--Schwarz estimate in the kernel variables.  Passing to the limit proves the formula.
\end{proof}

For $q\ge2$, let $C_q(B)$ be the cyclic amalgam obtained from copies $B_1,\ldots,B_q$ by identifying the right root of $B_i$ with the left root of $B_{i+1}$, with indices read modulo $q$.  We work with simple graphs and assume that these identifications create no parallel edges.  If $q=2$ and $\ell r\in E(B)$, the two copies of the root edge are identified with the same unordered vertex pair and produce a double edge; this is why statements involving $C_2(B)$ assume that the roots are non-adjacent.

The trace factorization gives the following operator-valued criterion.

\begin{theorem}[Cyclic birooted Schatten criterion]\label{thm:cyclic-schatten-criterion}
Let $(B,\ell,r)$ be a birooted graph with $b=e(B)\ge1$, and let $q\ge2$.  Assume that the identifications defining $C_q(B)$ create no parallel edges; in particular, if $q=2$, assume that $\ell r\notin E(B)$.
\begin{enumerate}[label=\textup{(\roman*)}]
\item Suppose that $\mathsf T_B^W$ is positive semidefinite for every bounded symmetric real kernel $W$.  If
\begin{equation}\label{eq:cyclic-local-certificate}
        \|\mathsf T_S^W\|_{S_q}
        \le
        \|\mathsf T_B^W\|_{S_q}^{e(S)/b}
\end{equation}
for every non-empty birooted subgraph $S\subseteq B$ and every such $W$, then $C_q(B)$ is strongly dominating and
\[
        p_{C_q(B)}(W)=\|\mathsf T_B^W\|_{S_q}^{1/b}.
\]
\item Suppose that $q$ is even and $B$ has an automorphism interchanging its two roots.  Then $C_q(B)$ is strongly dominating if and only if \eqref{eq:cyclic-local-certificate} holds for every non-empty $S$ and every $W$.  The same formula for $p_{C_q(B)}$ holds.
\end{enumerate}
\end{theorem}

\begin{proof}
If a subgraph $F\subseteq C_q(B)$ restricts to birooted subgraphs $S_1,\ldots,S_q$ in the consecutive blocks, then Lemma~\ref{lem:cyclic-trace-formula} gives
\begin{equation}\label{eq:cyclic-trace-factorization}
        t_F(W)
        =\tr\bigl(\mathsf T_{S_1}^W\cdots \mathsf T_{S_q}^W\bigr).
\end{equation}
The Schatten H\"older inequality \cite{SimonTraceIdeals} gives
\[
        |t_F(W)|
        \le \prod_{i=1}^q\|\mathsf T_{S_i}^W\|_{S_q}.
\]
Empty restrictions are retained as averaging projections in the operator product.  Each such projection has $S_q$-norm $1$, so only its numerical norm factor and its zero edge contribution are suppressed in the H\"older bound; the projection itself is not deleted from the trace product.  Thus the hypotheses for non-empty restrictions suffice, and no zero-exponent convention is needed.

Under the hypothesis in part~(i),
\[
        t_{C_q(B)}(W)
        =\tr\bigl((\mathsf T_B^W)^q\bigr)
        =\|\mathsf T_B^W\|_{S_q}^q.
\]
Combining this identity, \eqref{eq:cyclic-local-certificate}, and \eqref{eq:cyclic-trace-factorization} proves sufficiency and the formula for the full density.

For part~(ii), let $\phi$ be an automorphism interchanging the two roots.  The adjoint identity gives $\mathsf T_{B^{\mathrm{rev}}}^W=(\mathsf T_B^W)^*$, while $\phi$ identifies $B^{\mathrm{rev}}$ with $B$ as a birooted graph.  Hence $\mathsf T_B^W=(\mathsf T_B^W)^*$, so the full transfer is self-adjoint.  Since $q$ is even,
\[
        \tr\bigl((\mathsf T_B^W)^q\bigr)=\|\mathsf T_B^W\|_{S_q}^q,
\]
so the same argument proves sufficiency.  For necessity, fix $S\subseteq B$.  The image $\phi(S)$ is an embedded copy of $S$ with the boundary labels exchanged, and the transfer identity is
\[
        \mathsf T_{\phi(S)}^W
        =\mathsf T_{S^{\mathrm{rev}}}^W
        =(\mathsf T_S^W)^*.
\]
Place $S$ and $\phi(S)$ alternately around the cycle.  The resulting subgraph has density
\[
        \tr\bigl((\mathsf T_S^W(\mathsf T_S^W)^*)^{q/2}\bigr)
        =\|\mathsf T_S^W\|_{S_q}^q
\]
and has $q e(S)$ edges.  Domination by the full cyclic amalgam is therefore exactly \eqref{eq:cyclic-local-certificate}.
\end{proof}

The preceding criterion can be reversed simultaneously over all even exponents, giving an analytic characterization of the blocks which generate an infinite even cyclic family by the Schatten mechanism.

The universal even-Schatten condition was defined in the introduction.  We now prove the converse across all even exponents.

\begin{proof}[Proof of Theorem~\ref{thm:exact-cyclic-source}]
If (i) holds, Theorem~\ref{thm:cyclic-schatten-criterion}(ii) gives (ii) for every even $q$.

Conversely, assume (ii), and fix an even $q$, a non-empty birooted subgraph $S\subseteq B$, and a kernel $W$.  Let $\phi$ be a root-swapping automorphism.  Then $\phi(S)$ is an embedded copy with the boundary labels exchanged and
\[
        \mathsf T_{\phi(S)}^W
        =\mathsf T_{S^{\mathrm{rev}}}^W
        =(\mathsf T_S^W)^*.
\]
Place $S$ and $\phi(S)$ alternately in the $q$ blocks.  By this identity, the resulting subgraph $F\subseteq C_q(B)$ has $qe(S)$ edges and
\[
        t_F(W)=\tr\!\left((\mathsf T_S^W(\mathsf T_S^W)^*)^{q/2}\right)
              =\|\mathsf T_S^W\|_{S_q}^{q}.
\]
Root reversal makes $\mathsf T_B^W$ self-adjoint, and hence
\[
        t_{C_q(B)}(W)=\tr((\mathsf T_B^W)^q)
                     =\|\mathsf T_B^W\|_{S_q}^{q}.
\]
Strong domination of $F$ gives \eqref{eq:intro-universal-schatten} at exponent $q$.  Since every even exponent occurs, (i) follows.  The final assertion is Theorem~\ref{thm:cyclic-schatten-criterion}(i), applied separately at each exponent.  Finally, \eqref{eq:intro-moment-formulation} is simply the identity
\[
        \|A\|_{S_{2m}}=\tr(|A|^{2m})^{1/(2m)}
\]
written after clearing the positive powers in the local inequality.
\end{proof}

\begin{remark}\label{rem:source-theorem-scope}
Theorem~\ref{thm:exact-cyclic-source} is an analytic characterization rather than an enumeration of the finite combinatorial shapes of all admissible blocks.  It converts that remaining problem into a local spectral one.  In particular, failure at one even spectral moment produces an explicit cyclic amalgam which is not strongly dominating.
\end{remark}

\subsection{The norming boundary block}

Before turning to the non-norming block, we isolate the basic source which reconnects the transfer-operator criterion with reflection-graph norms.  Let $\mathcal Q$ be a copy of $K_{2,2}$ whose two roots $\ell,r$ form one bipartition class, and write $a,b$ for the other two vertices.

\begin{theorem}[Uniform certificate for the two-rooted $K_{2,2}$ block]\label{thm:birooted-k22-certificate}
For every integer $p\ge2$, every bounded symmetric real kernel $W$, and every non-empty birooted subgraph $S\subseteq\mathcal Q$,
\begin{equation}\label{eq:k22-schatten-local}
        \|\mathsf T_S^W\|_{S_p}
        \le
        \|\mathsf T_{\mathcal Q}^W\|_{S_p}^{e(S)/4}.
\end{equation}
Moreover, $\mathsf T_{\mathcal Q}^W$ is positive semidefinite for every $W$.  Consequently $C_q(\mathcal Q)$ is strongly dominating for every integer $q\ge2$.
\end{theorem}

\begin{proof}
Let $A=A_W$, and let
\[
        K(x,y)=\int_0^1W(x,z)W(y,z)\,dz
\]
be the kernel of the positive semidefinite operator $A^2$.  The full transfer kernel is
\[
        L(x,y)=K(x,y)^2,
\]
so $\mathsf L:=\mathsf T_{\mathcal Q}^W=K\circ K$ is positive semidefinite by the Schur product theorem.  Put
\[
        P=\|\mathsf L\|_{S_p}^{1/4}.
\]
For $\sigma\subseteq\{a,b\}$, let $X_\sigma:L^2([0,1]^2)\to L^2([0,1])$ have kernel
\[
        A_\sigma(x;u,v)=\prod_{c\in\sigma}W(x,\xi_c),
        \qquad \xi_a=u,\quad \xi_b=v.
\]
We claim that
\begin{equation}\label{eq:k22-half-bound}
        \|X_\sigma\|_{S_{2p}}\le P^{|\sigma|}.
\end{equation}
The cases $|\sigma|=0$ and $|\sigma|=2$ are immediate: $X_\varnothing$ has rank one and norm $1$, while
\[
        X_{\{a,b\}}X_{\{a,b\}}^*=\mathsf L,
        \qquad
        \|X_{\{a,b\}}\|_{S_{2p}}^2=\|\mathsf L\|_{S_p}=P^4.
\]
If $|\sigma|=1$, then $X_\sigma X_\sigma^*=K$.  Since $K$ is positive semidefinite, the cycle form of Cauchy--Schwarz gives
\begin{align*}
        \|K\|_{S_p}^{2p}
        &=\tr(K^p)^2\\
        &=\left(\int\prod_{i=1}^pK(x_i,x_{i+1})\,d\mathbf x\right)^2\\
        &\le\int\prod_{i=1}^pK(x_i,x_{i+1})^2\,d\mathbf x\\
        &=\tr((K\circ K)^p)
         =\|\mathsf L\|_{S_p}^p,
\end{align*}
where $x_{p+1}=x_1$.  Hence
\[
        \|X_\sigma\|_{S_{2p}}^2
        =\|K\|_{S_p}
        \le\|\mathsf L\|_{S_p}^{1/2}=P^2,
\]
which proves \eqref{eq:k22-half-bound}.

Every subgraph $S\subseteq\mathcal Q$ is specified by the sets $\sigma_\ell,\sigma_r\subseteq\{a,b\}$ of retained edges at its two roots, and
\[
        \mathsf T_S^W=X_{\sigma_\ell}X_{\sigma_r}^*,
        \qquad
        e(S)=|\sigma_\ell|+|\sigma_r|.
\]
Schatten H\"older and \eqref{eq:k22-half-bound} now give
\[
        \|\mathsf T_S^W\|_{S_p}
        \le
        \|X_{\sigma_\ell}\|_{S_{2p}}
        \|X_{\sigma_r}\|_{S_{2p}}
        \le P^{e(S)},
\]
which is \eqref{eq:k22-schatten-local}.  The final assertion follows from Theorem~\ref{thm:cyclic-schatten-criterion}(i).
\end{proof}

For $n\ge2$, let $R_n$ be the bipartite graph with junction vertices $x_i$ for $i\in\mathbb Z/n\mathbb Z$ and two further vertices $y_i^0,y_i^1$ adjacent precisely to $x_{i-1}$ and $x_i$ for each $i$.  Thus $R_2\cong K_{2,4}$, while for $n\ge3$ the graph $R_n$ is the $K_{2,2}$-replacement of the cycle $C_n$.  By construction,
\[
        C_q(\mathcal Q)\cong R_q.
\]

\begin{proposition}[Reflection identification of the doubled cycles]\label{prop:Rn-reflection}
For every integer $n\ge2$, the graph $R_n$ is seminorming and, in the present non-degenerate terminology, norming.
\end{proposition}

\begin{proof}
We recall the relevant definition.  Let $\Gamma$ be a finite reflection group with a fixed set $S$ of simple reflections, and let $S_1,S_2\subseteq S$.  Write $\Gamma_i=\langle S_i\rangle$.  The $(S_1,S_2;S,\Gamma)$-reflection graph is the bipartite graph with parts the left cosets $\Gamma/\Gamma_1$ and $\Gamma/\Gamma_2$, where, for each $g\in\Gamma$, the two cosets $g\Gamma_1$ and $g\Gamma_2$ are adjacent.  Repeated incidences give only one edge.

Now let
\[
        I_2(n)=\langle s,t:s^2=t^2=(st)^n=1\rangle
\]
be the dihedral reflection group, let $A_1=\langle u:u^2=1\rangle$, and take
\[
        \Gamma=I_2(n)\times A_1
\]
with simple reflections $S=\{s,t,u\}$.  Put
\[
        S_1=\{s,u\},
        \qquad
        S_2=\{t\},
        \qquad
        \Gamma_1=\langle s,u\rangle,
        \qquad
        \Gamma_2=\langle t\rangle.
\]
Since $|\Gamma|=4n$, $|\Gamma_1|=4$, and $|\Gamma_2|=2$, the two parts contain $n$ and $2n$ vertices, respectively.  Writing $\rho=st$, index the cosets by
\[
        X_i=\rho^i\Gamma_1,
        \qquad
        Y_i^\varepsilon=\rho^iu^\varepsilon\Gamma_2,
        \qquad i\in\mathbb Z/n\mathbb Z,
        \quad \varepsilon\in\{0,1\}.
\]
The coset $Y_i^\varepsilon$ consists of the two elements
\[
        \rho^iu^\varepsilon
        \quad\text{and}\quad
        \rho^iu^\varepsilon t.
\]
The first lies in $X_i$.  Since $t=\rho^{-1}s$ and $s,u\in\Gamma_1$, the second lies in
\[
        \rho^{i-1}\Gamma_1=X_{i-1}.
\]
These are the only two incidences of $Y_i^\varepsilon$, so it is adjacent precisely to $X_i$ and $X_{i-1}$.  Hence the reflection graph is exactly $R_n$.  Finally, $S_1\cap S_2=\varnothing$, so Conlon and Lee's reflection-graph theorem gives the seminorming property in our terminology \cite[Theorem~1.3]{ConlonLeeReflection}.  Lee and Sidorenko proved that a graph with no isolated vertices can be seminorming without being norming only when it is a disjoint union of isomorphic even stars (including copies of $K_2$) \cite[Theorem~4.1]{LeeSidorenkoComplex}.  The graph $R_n$ is connected and is not a star, so its seminorm is non-degenerate and $R_n$ is norming.
\end{proof}

\section{Complete-bipartite cyclic sources}

For $m\ge2$, let $\mathcal B_m$ be a copy of $K_{2,m}$ whose two roots $\ell,r$ are distinct vertices in the part of size $m$.  Write $a,b$ for the two vertices in the other part.  The case $m=2$ is the norming source $\mathcal Q$ above.  We now determine the source condition for the entire family $(\mathcal B_m)_{m\ge2}$.

The even case is organized around a fixed split of the $2t-2$ private vertices into two groups of size $t-1$.  A subgraph is encoded by its retained root edges and by the four possible states of each private vertex.  This produces two half-operators whose product is the full transfer.  We first isolate the analytic tools, then collect the notation and the state decomposition before proving the certificate.

\subsection{Analytic tools}

We first separate the two notions of positivity used in that argument.

\begin{definition}[Gram-kernel and operator order]\label{def:kernel-operator-order}
A symmetric measurable kernel $Q$ on a probability space is \emph{Gram-positive} if there are a Hilbert space $\mathcal H$ and a measurable map $\Phi$ such that
\[
        Q(u,v)=\langle\Phi(u),\Phi(v)\rangle_{\mathcal H}
\]
for almost every $(u,v)$.  We write $Q\preceqG R$ when $R-Q$ is Gram-positive.  A bounded Gram-positive kernel induces a positive semidefinite integral operator.  For bounded self-adjoint operators we write $\mathsf A\preceqop\mathsf B$ for the usual quadratic-form order.  The implication from Gram-kernel order to operator order will be used below; the converse is neither asserted nor needed.
\end{definition}

We use the following standard product-space form of Finner's inequality.

\begin{lemma}[Uniform-coordinate Finner inequality]\label{lem:finner-coordinate}
Let $(\Omega_i,\mu_i)_{i\in I}$ be probability spaces.  For each $j\in J$, let $g_j\ge0$ be a measurable function of the coordinates indexed by a set $I_j\subseteq I$.  If every coordinate belongs to at most $r$ of the sets $I_j$, counted with multiplicity, then
\begin{equation}\label{eq:finner-coordinate}
        \int_{\prod_{i\in I}\Omega_i}\prod_{j\in J}g_j(x_{I_j})\,d\mu(x)
        \le
        \prod_{j\in J}\|g_j\|_{L^r}.
\end{equation}
\end{lemma}

\begin{proof}
This is the fractional-matching form of Finner's generalized H\"older inequality, with weight $1/r$ on every factor \cite{Finner}.  Repeated copies of a factor are treated as distinct members of $J$.
\end{proof}

The proof of the even case uses four elementary operator facts.  We state them separately because each marks a distinct interface in the argument.

\begin{lemma}[Schur powers preserve Gram-kernel order]\label{lem:schur-power-order}
Let $D$ and $K$ be Gram-positive kernels with $D\preceqG K$.  Then, for every integer $r\ge1$,
\[
        D^{\circ r}\preceqG K^{\circ r}.
\]
More generally, if $Q\preceqG R$ and $P$ is Gram-positive, then
\[
        P\circ Q\preceqG P\circ R.
\]
Here $\circ$ denotes pointwise, or Schur, product.  No pointwise sign assumption on the kernels is required.
\end{lemma}

\begin{proof}
The Schur product theorem says that the pointwise product of two Gram-positive kernels is again Gram-positive.  Hence
\[
        P\circ R-P\circ Q=P\circ(R-Q)
\]
is Gram-positive, proving the second assertion.  For the first, use the telescoping identity
\[
 K^{\circ r}-D^{\circ r}
 =\sum_{j=0}^{r-1}
   D^{\circ j}\circ(K-D)\circ K^{\circ(r-1-j)}.
\]
Each summand is Gram-positive by repeated use of the Schur product theorem, so the difference is Gram-positive.
\end{proof}

\begin{lemma}[Truncated Schatten interpolation]\label{lem:truncated-schatten-interpolation}
Let $(\Omega,\mu)$ and $(Y,\nu)$ be finite measure spaces, let $G\in L^\infty(Y\times\Omega)$, and let $K\in L^\infty(\Omega)$ be real valued.  Fix an integer $t\ge2$ and an exponent $2\le r<\infty$.  For $1\le s\le t$, let $X_s:L^2(\Omega)\to L^2(Y)$ be the integral operator with kernel
\[
        G(y,\omega)K(\omega)^{s-1}.
\]
Then
\begin{equation}\label{eq:abstract-schatten-interpolation}
        \|X_s\|_{S_r}
        \le
        \|X_1\|_{S_r}^{(t-s)/(t-1)}
        \|X_t\|_{S_r}^{(s-1)/(t-1)}.
\end{equation}
\end{lemma}

\begin{proof}
The endpoint cases are immediate, so assume $1<s<t$ and put
\[
        \theta=\frac{s-1}{t-1}.
\]
For $n\ge1$, let
\[
        E_n=\{\omega\in\Omega:n^{-1}\le |K(\omega)|\le n\}
\]
and let $P_n$ be multiplication by $1_{E_n}$ on $L^2(\Omega)$.  On the closed strip $0\le\Re z\le1$, define $X_n(z)$ by the kernel
\[
        G(y,\omega)1_{E_n}(\omega)|K(\omega)|^{(t-1)z}.
\]
Because $G$ and $K$ are bounded and the measure spaces are finite, every $X_s$ is Hilbert--Schmidt.  On $E_n$ the function $|\log|K||$ is bounded, so $z\mapsto X_n(z)$ is analytic as an $S_2$-valued family in the open strip and continuous on its closure.

For real $\eta$, let $U_{0,n,\eta}$ be multiplication by $|K|^{i\eta(t-1)}$ on $E_n$ and by $1$ outside $E_n$.  This is unitary, and
\[
        X_n(i\eta)=X_1P_nU_{0,n,\eta}.
\]
Right multiplication by a unitary preserves singular values, while $P_n$ is a contraction; hence
\[
        \|X_n(i\eta)\|_{S_r}\le \|X_1\|_{S_r}.
\]
Likewise, let $U_{1,n,\eta}$ be multiplication by
\[
        \operatorname{sgn}(K)^{t-1}|K|^{i\eta(t-1)}
\]
on $E_n$ and by $1$ outside $E_n$.  Then
\[
        X_n(1+i\eta)=X_tP_nU_{1,n,\eta},
        \qquad
        \|X_n(1+i\eta)\|_{S_r}\le \|X_t\|_{S_r}.
\]

We now spell out the interpolation step.  Let $r'=r/(r-1)$ and take
$A\in S_{r'}(L^2(Y),L^2(\Omega))$.  Since $r'\le2$, the operator $A$ is Hilbert--Schmidt, and the scalar function
\[
        f_A(z)=\operatorname{tr}(A X_n(z))
\]
is analytic in the open strip and continuous on its closure.  Schatten H\"older and the two boundary estimates give
\[
\begin{aligned}
        |f_A(i\eta)|
        &\le \|A\|_{S_{r'}}\|X_1\|_{S_r},\\
        |f_A(1+i\eta)|
        &\le \|A\|_{S_{r'}}\|X_t\|_{S_r}.
\end{aligned}
\]
The scalar three-lines theorem therefore yields
\[
        |f_A(\theta)|
        \le
        \|A\|_{S_{r'}}
        \|X_1\|_{S_r}^{1-\theta}
        \|X_t\|_{S_r}^{\theta}.
\]
Taking the supremum over the unit ball of $S_{r'}$, using
$(S_r)^*=S_{r'}$ under the trace pairing, gives
\begin{equation}\label{eq:truncated-interior-bound}
        \|X_n(\theta)\|_{S_r}
        \le
        \|X_1\|_{S_r}^{1-\theta}
        \|X_t\|_{S_r}^{\theta}.
\end{equation}
Thus the proof uses only the scalar three-lines theorem and Schatten duality, rather than an unstated vector-valued interpolation principle.

If $U_{s,n}$ is multiplication by $\operatorname{sgn}(K)^{s-1}$ on $E_n$ and by $1$ outside $E_n$, then $U_{s,n}$ is unitary and
\[
        X_sP_n=X_n(\theta)U_{s,n}.
\]
Again, right multiplication by a unitary preserves singular values, so the bound in \eqref{eq:truncated-interior-bound} also holds for $X_sP_n$.

Finally, $P_n\uparrow1_{\{K\ne0\}}$ pointwise.  Since $s>1$, the kernel of $X_s$ vanishes on $\{K=0\}$, and dominated convergence gives $X_sP_n\to X_s$ in $S_2$.  For $r\ge2$, the singular-value inequality $\|T\|_{S_r}\le\|T\|_{S_2}$ gives the continuous inclusion $S_2\hookrightarrow S_r$, so the same convergence holds in $S_r$.  Letting $n\to\infty$ proves \eqref{eq:abstract-schatten-interpolation}.
\end{proof}

\begin{lemma}[Pointwise comparison from a quadratic-form order]\label{lem:quadratic-pointwise-comparison}
Let $Q_1\preceqop Q_2$ be positive semidefinite operators on $L^2(\Omega)$, and let $h_{x,y}\in L^2(\Omega)$ be any measurable family of test functions.  Define
\[
        M_i(x,y)=\langle h_{x,y},Q_i h_{x,y}\rangle
        \qquad(i=1,2).
\]
Then, for almost every $(x,y)$,
\[
        0\le M_1(x,y)\le M_2(x,y).
\]
\end{lemma}

\begin{proof}
Positivity of $Q_1$ gives $M_1(x,y)\ge0$.  Applying the positive operator $Q_2-Q_1$ to the particular test vector $h_{x,y}$ gives
\[
        M_2(x,y)-M_1(x,y)
        =\langle h_{x,y},(Q_2-Q_1)h_{x,y}\rangle\ge0.
\]
The statement is therefore a consequence of the special quadratic-form representation; it is not an assertion that arbitrary operator order implies pointwise order of arbitrary kernels.
\end{proof}

\begin{lemma}[Integer Schatten monotonicity for non-negative kernels]\label{lem:integer-schatten-monotonicity}
Let $M$ and $L$ be positive semidefinite integral operators with bounded pointwise non-negative kernels $m$ and $\ell$.  If
\[
        0\le m(x,y)\le \ell(x,y)
        \quad\text{for almost every }(x,y),
\]
then, for every integer $p\ge2$,
\[
        \|M\|_{S_p}\le\|L\|_{S_p}.
\]
\end{lemma}

\begin{proof}
The bounded kernels lie in $L^2$, so $M$ and $L$ are Hilbert--Schmidt.  Hence $M^p$ and $L^p$ are trace class for every integer $p\ge2$.  Since the operators are positive semidefinite,
\[
        \|M\|_{S_p}^p=\operatorname{tr}(M^p),
        \qquad
        \|L\|_{S_p}^p=\operatorname{tr}(L^p).
\]
Lemma~\ref{lem:cyclic-trace-formula}, applied with all factors equal, gives the cycle-kernel formula; thus, with $x_{p+1}=x_1$,
\[
\begin{aligned}
        \operatorname{tr}(M^p)
        &=\int_{[0,1]^p}\prod_{i=1}^p m(x_i,x_{i+1})\,d\mathbf x\\
        &\le\int_{[0,1]^p}\prod_{i=1}^p \ell(x_i,x_{i+1})\,d\mathbf x
         =\operatorname{tr}(L^p)
         =\|L\|_{S_p}^p.
\end{aligned}
\]
Taking $p$th roots proves the claim.  The integrality of $p$ and the non-negativity of the two kernels are essential to this argument.
\end{proof}

\subsection{Half-state encoding and the even certificate}

The notation used in the proof is collected in Table~\ref{tab:k2even-notation}.  Lower-case Roman letters denote scalar functions or kernels, while sans-serif letters denote the associated integral operators.  The only matrix objects appearing below are finite-rank step-kernel approximants in Lemma~\ref{lem:cyclic-trace-formula}; they are not identified with the limiting kernels.

\begin{table}[t]
\centering
\small
\renewcommand{\arraystretch}{1.15}
\begin{tabularx}{\textwidth}{@{}l l Y@{}}
\toprule
symbol & type & meaning \\
\midrule
$x,y$ & root variables & output variables of a transfer operator \\
$u,v$ & centre variables & variables assigned to the two vertices $a,b$ in the size-$2$ part \\
$d(u)$ & scalar function & row average $\int W(u,z)\,dz$ \\
$K(u,v)$ & Gram kernel & $\int W(u,z)W(v,z)\,dz$ \\
$D(u,v)$ & rank-one Gram kernel & $d(u)d(v)$, with $D\preceqG K$ \\
$\sigma$ & subset of $\{a,b\}$ & retained edges from one root to the two centres \\
$\tau_j$ & subset of $\{a,b\}$ & retained edges from the $j$th private vertex to the centres \\
$\alpha,\beta,\gamma$ & non-negative integers & numbers of states $\{a\}$, $\{b\}$ and $\{a,b\}$ in one half \\
$h$ & integer & retained edge count $|\sigma|+\alpha+\beta+2\gamma$ in one half \\
$\mathsf X_{\sigma,\boldsymbol\tau}$ & $L^2([0,1]^2)\to L^2([0,1])$ & half-transfer operator associated with one root and $t-1$ private vertices \\
$\mathsf X_s$ & $L^2([0,1]^2)\to L^2([0,1])$ & balanced two-root half-operator with kernel $W(x,u)W(x,v)K(u,v)^{s-1}$ \\
$\mathsf Y_{\alpha,\beta,\gamma}$ & $L^2([0,1]^2)\to L^2([0,1])$ & auxiliary two-root half-operator for an exponent pattern \\
$\mathsf M_{\alpha,\beta,\gamma}$ & positive trace-class operator & $\mathsf Y_{\alpha,\beta,\gamma}\mathsf Y_{\alpha,\beta,\gamma}^{*}$ \\
$\mathsf L_t$ & positive trace-class operator & full transfer $\mathsf T_{\mathcal B_{2t}}^W=\mathsf X_t\mathsf X_t^*$ \\
$P_0,P$ & non-negative scalars & endpoint and full-transfer normalizations in \eqref{eq:k2even-P-P0} \\
\bottomrule
\end{tabularx}
\caption{Notation for the $K_{2,2t}$ certificate.}
\label{tab:k2even-notation}
\end{table}

\begin{figure}[t]
\centering
\begin{tikzpicture}[scale=0.88, every node/.style={font=\small}]
  \begin{scope}[xshift=-4.7cm]
    \node[circle,draw,inner sep=2pt] (x) at (-1.4,0) {$x$};
    \node[circle,draw,inner sep=2pt] (a) at (0,0.8) {$a$};
    \node[circle,draw,inner sep=2pt] (b) at (0,-0.8) {$b$};
    \draw[dashed] (x)--(a);
    \draw[dashed] (x)--(b);
    \node[font=\scriptsize,align=center] at (-0.75,1.3) {root edges\\selected by $\sigma$};

    \node[circle,draw,inner sep=1.8pt] (ze) at (1.55,1.55) {$z_\varnothing$};
    \node[circle,draw,inner sep=1.8pt] (za) at (2.6,0.55) {$z_a$};
    \node[circle,draw,inner sep=1.8pt] (zb) at (2.6,-0.55) {$z_b$};
    \node[circle,draw,inner sep=1.8pt] (zab) at (1.55,-1.55) {$z_{ab}$};
    \draw (a)--(za);
    \draw (b)--(zb);
    \draw (a)--(zab);
    \draw (b)--(zab);
    \node[font=\scriptsize,align=center] at (1.7,0) {private states\\$\tau\in\{\varnothing,\{a\},\{b\},\{a,b\}\}$};
    \node at (0.6,-2.15) {(a) one half-state};
  \end{scope}

  \begin{scope}[xshift=3.3cm]
    \node[align=left] at (0,1.25) {\textbf{Diagonal pattern:}\quad $\alpha=\beta=k$};
    \node[align=left] at (0,0.55) {$d(u)^k d(v)^k K(u,v)^\gamma$};
    \draw[-{Latex[length=2mm]},thick] (-0.1,0.15)--(-0.1,-0.45);
    \node[align=left] at (0,-0.85) {$\mathsf M_{k,k,\gamma}$ is compared with $\mathsf L_{1+k+\gamma}$.};

    \node[align=left] at (0,-1.65) {\textbf{Off-diagonal pattern:}\quad $\alpha\ne\beta$};
    \node[align=left] at (0,-2.35) {$\mathsf M_{\alpha,\beta,\gamma}$ is reduced by cycle Cauchy--Schwarz};
    \node[align=left] at (0,-2.95) {to $\mathsf M_{\alpha,\alpha,\gamma}$ and $\mathsf M_{\beta,\beta,\gamma}$.};
  \end{scope}
\end{tikzpicture}
\caption{The half-state encoding.  Dashed root edges indicate the choice of $\sigma$; the four private-vertex states contribute $1,d(u),d(v)$ and $K(u,v)$, respectively.  For the two-root boundary state, equal one-sided exponents give the diagonal case, while unequal exponents are reduced to two diagonal cases by a cycle Cauchy--Schwarz inequality.}
\label{fig:k2even-half-states}
\end{figure}

\paragraph{Proof map.}
The proof of Theorem~\ref{thm:birooted-k2even-certificate} separates the four possible root-edge states of one half.  States with no root edge or one root edge are controlled directly by Finner's inequality after the coordinate multiplicities are listed.  For the two-root-edge state, the balanced half-operators are first interpolated between their endpoints.  Gram-kernel order and Lemma~\ref{lem:integer-schatten-monotonicity} then control the diagonal exponent patterns, while a cycle Cauchy--Schwarz argument reduces every off-diagonal pattern to two diagonal ones.  Finally, the two half-state estimates are assembled through
\[
        \mathsf T_S^W
        =\mathsf X_{\sigma_\ell,\boldsymbol\tau^L}
         \mathsf X_{\sigma_r,\boldsymbol\tau^R}^{*}
\]
and Schatten H\"older.  The cycle expansion is used only for integer exponents $p\ge2$; no monotonicity statement for arbitrary real Schatten exponents is asserted.

\begin{theorem}[Uniform certificate for the two-rooted $K_{2,2t}$ blocks]\label{thm:birooted-k2even-certificate}
Let $t\ge2$.  For every integer $p\ge2$, every bounded symmetric real kernel $W$, and every non-empty birooted subgraph $S\subseteq\mathcal B_{2t}$,
\begin{equation}\label{eq:k2even-schatten-local}
        \|\mathsf T_S^W\|_{S_p}
        \le
        \|\mathsf T_{\mathcal B_{2t}}^W\|_{S_p}^{e(S)/(4t)}.
\end{equation}
Moreover, $\mathsf T_{\mathcal B_{2t}}^W$ is positive semidefinite for every $W$.
\end{theorem}

The conclusion is stronger than universality: the universal even-Schatten condition only requires even exponents, whereas the theorem proves the local inequality for every integer $p\ge2$.

\begin{proof}
Fix once and for all a split of the $2t-2$ private vertices into labelled sets $Z_L$ and $Z_R$, each of size $t-1$.  For a subgraph $S\subseteq\mathcal B_{2t}$, let $\sigma_\ell,\sigma_r\subseteq\{a,b\}$ record the retained edges at the two roots, and let $\boldsymbol\tau^L,\boldsymbol\tau^R$ record the retained incident edges at the private vertices in $Z_L,Z_R$.  The half-operators defined below satisfy
\begin{equation}\label{eq:k2even-architecture-factorization}
        \mathsf T_S^W
        =\mathsf X_{\sigma_\ell,\boldsymbol\tau^L}
         \mathsf X_{\sigma_r,\boldsymbol\tau^R}^{*},
\end{equation}
and their half-edge counts add to $e(S)$.  It is therefore enough to prove a uniform Schatten bound for one arbitrary half-state.

\medskip
\noindent\emph{Step 1: full transfer and normalization.}
Put
\[
        d(u)=\int_0^1W(u,z)\,dz,
        \qquad
        K(u,v)=\int_0^1W(u,z)W(v,z)\,dz.
\]
The kernel $K$ is Gram-positive, with feature map $u\mapsto W(u,\cdot)$, and its integral operator is $A_W^2$.  Let $\mathsf X_t:L^2([0,1]^2)\to L^2([0,1])$ be the full half-operator with kernel
\begin{equation}\label{eq:k2even-full-half}
        \Phi_t(x;u,v)=W(x,u)W(x,v)K(u,v)^{t-1}.
\end{equation}
The full transfer operator $\mathsf L_t:=\mathsf T_{\mathcal B_{2t}}^W$ has kernel
\begin{equation}\label{eq:k2even-full-transfer}
        L_t(x,y)
        =\int W(x,u)W(x,v)W(y,u)W(y,v)K(u,v)^{2t-2}\,du\,dv,
\end{equation}
so
\[
        \mathsf L_t=\mathsf X_t\mathsf X_t^*\succeqop0.
\]
Set
\begin{equation}\label{eq:k2even-P-P0}
        P=\|\mathsf L_t\|_{S_p}^{1/(4t)},
        \qquad
        P_0=\left(\int_{[0,1]^2}|K(u,v)|^{2t}\,du\,dv\right)^{1/(4t)}.
\end{equation}
Integrating \eqref{eq:k2even-full-transfer} first in $x$ and $y$ gives
\[
        P_0^{4t}=\langle\one,\mathsf L_t\one\rangle
        \le\|\mathsf L_t\|_{\mathrm{op}}
        \le\|\mathsf L_t\|_{S_p}=P^{4t},
\]
and hence
\begin{equation}\label{eq:k2even-P0P}
        P_0\le P.
\end{equation}

Let $D(u,v)=d(u)d(v)$.  Both $D$ and $K-D$ have explicit Gram representations:
\[
        D(u,v)=\langle d(u)\one,d(v)\one\rangle_{L^2([0,1])},
\]
and
\[
        K(u,v)-D(u,v)
        =\int_0^1\bigl(W(u,z)-d(u)\bigr)
                    \bigl(W(v,z)-d(v)\bigr)\,dz.
\]
Thus $D\preceqG K$.  Lemma~\ref{lem:schur-power-order} gives
\[
        D^{\circ 2t}\preceqG K^{\circ 2t}.
\]
The Gram-positive difference induces a positive semidefinite integral operator.  Testing that operator against $\one$ and using \eqref{eq:k2even-P-P0}, we obtain
\begin{equation}\label{eq:k2even-dKnorms}
        \|d\|_{L^{2t}}\le P_0,
        \qquad
        \|K\|_{L^{2t}([0,1]^2)}=P_0^2.
\end{equation}

\medskip
\noindent\emph{Step 2: a uniform half-operator comparison.}
For $\sigma\subseteq\{a,b\}$, define
\[
        A_\sigma(x;u,v)=\prod_{c\in\sigma}W(x,\xi_c),
        \qquad \xi_a=u,\quad \xi_b=v.
\]
For a private vertex, a state $\tau\subseteq\{a,b\}$ records which of its two incident edges are retained.  After integrating that vertex, put
\[
        c_\varnothing=1,
        \qquad c_{\{a\}}=d(u),
        \qquad c_{\{b\}}=d(v),
        \qquad c_{\{a,b\}}=K(u,v).
\]
For a private vertex $z_j$ in state $\tau_j$, direct integration gives
\[
        \int_0^1\prod_{c\in\tau_j}W(\xi_c,z_j)\,dz_j
        =c_{\tau_j}(u,v).
\]
The root variable contributes exactly $A_\sigma(x;u,v)$.  Hence, for states $\boldsymbol\tau=(\tau_1,\ldots,\tau_{t-1})$, the half-operator $\mathsf X_{\sigma,\boldsymbol\tau}$ has kernel
\begin{equation}\label{eq:k2even-general-half}
        A_\sigma(x;u,v)\prod_{j=1}^{t-1}c_{\tau_j}(u,v).
\end{equation}
Write
\[
\begin{aligned}
        \alpha&=|\{j:\tau_j=\{a\}\}|,\\
        \beta&=|\{j:\tau_j=\{b\}\}|,\\
        \gamma&=|\{j:\tau_j=\{a,b\}\}|,
\end{aligned}
\qquad
        h=|\sigma|+\alpha+\beta+2\gamma.
\]
The four boundary types and the estimates used below are summarized here.
\begin{center}
\small
\renewcommand{\arraystretch}{1.12}
\begin{tabularx}{0.98\textwidth}{@{}c c c Y c@{}}
\toprule
$\sigma$ & number of states & half-edge count $h$ & principal estimate & output \\
\midrule
$\varnothing$ & $1$ & $\alpha+\beta+2\gamma$ & rank one and Finner & $S_{2p}$ bound $P^h$ \\
$\{a\}$ or $\{b\}$ & $2$ & $1+\alpha+\beta+2\gamma$ & $S_4$ expansion and Finner & $S_{2p}$ bound $P^h$ \\
$\{a,b\}$ & $1$ & $2+\alpha+\beta+2\gamma$ & interpolation, then diagonal/off-diagonal reduction & $S_{2p}$ bound $P^h$ \\
\bottomrule
\end{tabularx}
\end{center}
We claim that
\begin{equation}\label{eq:k2even-half-bound}
        \|\mathsf X_{\sigma,\boldsymbol\tau}\|_{S_{2p}}
        \le P^h.
\end{equation}

\begin{claim}[Endpoint half-states]\label{clm:k2even-endpoint-states}
If $|\sigma|\le1$, then
\[
        \|\mathsf X_{\sigma,\boldsymbol\tau}\|_{S_{2p}}\le P^h.
\]
\end{claim}

\begin{proof}
Suppose first that $\sigma=\varnothing$.  The operator has rank one and its unique singular value is
\[
        \left\|d(u)^\alpha d(v)^\beta K(u,v)^\gamma\right\|_{L^2(du\,dv)}.
\]
After squaring, regard the integrand as a product of $2\alpha$ copies of $|d(u)|$, $2\beta$ copies of $|d(v)|$, and $2\gamma$ copies of $|K(u,v)|$.  The underlying coordinate hypergraph has vertex set $\{u,v\}$ and factor supports $\{u\}$, $\{v\}$, and $\{u,v\}$.  Its two coordinate degrees are
\[
        2\alpha+2\gamma\le2(t-1),
        \qquad
        2\beta+2\gamma\le2(t-1).
\]
Assigning exponent $2t$ to every factor in Lemma~\ref{lem:finner-coordinate} therefore gives
\[
\begin{aligned}
        \|\mathsf X_{\varnothing,\boldsymbol\tau}\|_{S_{2p}}^2
        &\le
        \|d\|_{L^{2t}}^{2\alpha+2\beta}
        \|K\|_{L^{2t}}^{2\gamma}\\
        &\le P_0^{2\alpha+2\beta+4\gamma}=P_0^{2h},
\end{aligned}
\]
where the last line uses \eqref{eq:k2even-dKnorms}.  Taking square roots and using \eqref{eq:k2even-P0P} proves the required bound.

Next suppose that $|\sigma|=1$, say $\sigma=\{a\}$.  Since $2p\ge4$, monotonicity of Schatten norms in the exponent gives
\[
        \|\mathsf X_{\sigma,\boldsymbol\tau}\|_{S_{2p}}
        \le\|\mathsf X_{\sigma,\boldsymbol\tau}\|_{S_4}.
\]
Expanding the fourth power yields
\begin{align*}
 &\|\mathsf X_{\{a\},\boldsymbol\tau}\|_{S_4}^4\\
 &\quad=\int K(u,u')^2
        d(u)^{2\alpha}d(u')^{2\alpha}
        d(v)^{2\beta}d(v')^{2\beta}
        K(u,v)^{2\gamma}K(u',v')^{2\gamma}
        \,du\,du'\,dv\,dv'.
\end{align*}
After absolute values are inserted, use the coordinate set $\{u,u',v,v'\}$.  There are two copies of $|K(u,u')|$, $2\alpha$ one-coordinate $d$-factors at each of $u,u'$, $2\beta$ at each of $v,v'$, and $2\gamma$ copies of each of $|K(u,v)|$ and $|K(u',v')|$.  Thus the coordinate degrees are
\[
        2(1+\alpha+\gamma)\le2t
        \quad\text{at }u,u',
        \qquad
        2(\beta+\gamma)\le2(t-1)<2t
        \quad\text{at }v,v'.
\]
Finner with exponent $2t$ and \eqref{eq:k2even-dKnorms} give
\[
\begin{aligned}
        \|\mathsf X_{\{a\},\boldsymbol\tau}\|_{S_4}^4
        &\le
        \|K\|_{L^{2t}}^{2+4\gamma}
        \|d\|_{L^{2t}}^{4\alpha+4\beta}\\
        &\le P_0^{4+4\alpha+4\beta+8\gamma}
         =P_0^{4h}.
\end{aligned}
\]
Taking fourth roots and using \eqref{eq:k2even-P0P} proves the claim; the case $\sigma=\{b\}$ is symmetric.
\end{proof}

It remains to take $\sigma=\{a,b\}$.

\begin{claim}[Balanced two-edge half-states]\label{clm:k2even-balanced}
For $1\le s\le t$, let $\mathsf X_s$ have kernel
\begin{equation}\label{eq:k2even-interpolation-half}
        W(x,u)W(x,v)K(u,v)^{s-1}.
\end{equation}
Then
\begin{equation}\label{eq:k2even-interpolation-bound}
        \|\mathsf X_s\|_{S_{2p}}\le P^{2s}
        \qquad(1\le s\le t).
\end{equation}
\end{claim}

\begin{proof}
At $s=t$ this is equality, because $\mathsf X_t\mathsf X_t^*=\mathsf L_t$.  At $s=1$,
\begin{align*}
        \|\mathsf X_1\|_{S_{2p}}^2
        &=\|K\circ K\|_{S_p}\\
        &\le\|K\circ K\|_{S_2}
         =\|K\|_{L^4}^2
         \le\|K\|_{L^{2t}}^2
         =P_0^4
         \le P^4.
\end{align*}
For $1<s<t$, apply Lemma~\ref{lem:truncated-schatten-interpolation} on
\[
        \Omega=[0,1]^2,
        \qquad
        G(x;u,v)=W(x,u)W(x,v),
\]
with the scalar multiplier $K(u,v)$, endpoint index $t$, and Schatten exponent $2p$.  If
\[
        \theta=\frac{s-1}{t-1},
\]
then the lemma gives
\[
        \|\mathsf X_s\|_{S_{2p}}
        \le
        \|\mathsf X_1\|_{S_{2p}}^{1-\theta}
        \|\mathsf X_t\|_{S_{2p}}^\theta
        \le P^{2(1-\theta)+2t\theta}=P^{2s}.
\]
This proves \eqref{eq:k2even-interpolation-bound}; the truncation, boundary estimates, and convergence in Schatten norm are all contained in Lemma~\ref{lem:truncated-schatten-interpolation}.
\end{proof}

For all non-negative integers $\alpha,\beta,\gamma$ satisfying
\begin{equation}\label{eq:k2even-extended-domain}
        \alpha+\gamma\le t-1,
        \qquad
        \beta+\gamma\le t-1,
\end{equation}
define an auxiliary half-operator $\mathsf Y_{\alpha,\beta,\gamma}:L^2([0,1]^2)\to L^2([0,1])$ with kernel
\[
        W(x,u)W(x,v)d(u)^\alpha d(v)^\beta K(u,v)^\gamma,
\]
and put
\[
        \mathsf M_{\alpha,\beta,\gamma}
        =\mathsf Y_{\alpha,\beta,\gamma}
         \mathsf Y_{\alpha,\beta,\gamma}^*.
\]
For an actual half-state with $\sigma=\{a,b\}$, the stronger condition
$\alpha+\beta+\gamma\le t-1$ holds and
$\mathsf Y_{\alpha,\beta,\gamma}=\mathsf X_{\{a,b\},\boldsymbol\tau}$.  We use the larger domain \eqref{eq:k2even-extended-domain} because the diagonal comparison naturally introduces the auxiliary triples $(\alpha,\alpha,\gamma)$ and $(\beta,\beta,\gamma)$.

Set $h_{xy}(u)=W(x,u)W(y,u)$.  The kernel of $\mathsf M_{\alpha,\beta,\gamma}$ is
\begin{equation}\label{eq:k2even-Mkernel}
 M_{\alpha,\beta,\gamma}(x,y)
 =\iint h_{xy}(u)h_{xy}(v)
 d(u)^{2\alpha}d(v)^{2\beta}K(u,v)^{2\gamma}\,du\,dv.
\end{equation}
The operators $\mathsf Y_{\alpha,\beta,\gamma}$ and $\mathsf X_s$ have bounded kernels on finite measure spaces, hence are Hilbert--Schmidt.  Consequently each $\mathsf M_{\alpha,\beta,\gamma}=YY^*$ and each $\mathsf L_s=\mathsf X_s\mathsf X_s^*$ is positive trace class, and all powers and traces below are well defined.

\begin{claim}[Diagonal two-edge boundary state]\label{clm:k2even-diagonal}
Let $k,\gamma\ge0$ satisfy $k+\gamma\le t-1$, and put $s=1+k+\gamma$.  Then
\[
        \|\mathsf M_{k,k,\gamma}\|_{S_p}\le P^{4s}.
\]
\end{claim}

\begin{proof}
If $k=\gamma=0$, then $M_{0,0,0}=L_1$ pointwise and the claim follows from \eqref{eq:k2even-interpolation-bound}.  Otherwise define Gram-positive kernels
\[
        Q_{k,\gamma}
        =(D\circ D)^{\circ k}\circ(K\circ K)^{\circ\gamma},
        \qquad
        R_{k,\gamma}
        =(K\circ K)^{\circ(k+\gamma)},
\]
with a factor of exponent zero omitted, and denote their integral operators by $\mathsf Q_{k,\gamma}$ and $\mathsf R_{k,\gamma}$.  Lemma~\ref{lem:schur-power-order}, followed by Schur multiplication with $(K\circ K)^{\circ\gamma}$, gives
\[
        Q_{k,\gamma}\preceqG R_{k,\gamma},
        \qquad
        \mathsf Q_{k,\gamma}\preceqop\mathsf R_{k,\gamma}.
\]
For the test function $h_{xy}(u)=W(x,u)W(y,u)$, formula \eqref{eq:k2even-Mkernel} becomes
\[
        M_{k,k,\gamma}(x,y)
        =\langle h_{xy},\mathsf Q_{k,\gamma}h_{xy}\rangle,
        \qquad
        L_s(x,y)
        =\langle h_{xy},\mathsf R_{k,\gamma}h_{xy}\rangle.
\]
Lemma~\ref{lem:quadratic-pointwise-comparison} therefore gives
\begin{equation}\label{eq:k2even-pointwise-order}
        0\le M_{k,k,\gamma}(x,y)\le L_s(x,y)
        \qquad\text{for almost every }(x,y).
\end{equation}
Both kernels are pointwise non-negative by these quadratic-form representations.  Lemma~\ref{lem:integer-schatten-monotonicity} and \eqref{eq:k2even-interpolation-bound} now yield
\begin{equation}\label{eq:k2even-diagonal-M}
        \|\mathsf M_{k,k,\gamma}\|_{S_p}
        \le\|\mathsf L_s\|_{S_p}
        =\|\mathsf X_s\|_{S_{2p}}^2
        \le P^{4s}.
\end{equation}
\end{proof}

\begin{claim}[Off-diagonal two-edge boundary state]\label{clm:k2even-offdiagonal}
Let $\alpha,\beta,\gamma\ge0$ satisfy \eqref{eq:k2even-extended-domain}.  Then
\begin{equation}\label{eq:k2even-offdiagonal-M}
        \|\mathsf M_{\alpha,\beta,\gamma}\|_{S_p}^2
        \le
        \|\mathsf M_{\alpha,\alpha,\gamma}\|_{S_p}
        \|\mathsf M_{\beta,\beta,\gamma}\|_{S_p}.
\end{equation}
Consequently, if $h=2+\alpha+\beta+2\gamma$, then
\[
        \|\mathsf M_{\alpha,\beta,\gamma}\|_{S_p}\le P^{2h}.
\]
\end{claim}

\begin{proof}
Let $\mathsf Q_\gamma$ be the positive semidefinite operator with Gram-positive kernel $K(u,v)^{2\gamma}$; when $\gamma=0$, this is the rank-one operator with constant kernel $1$.  Formula \eqref{eq:k2even-Mkernel} can be written as
\[
        M_{\alpha,\beta,\gamma}(x,y)
        =\langle h_{xy}d^{2\alpha},\mathsf Q_\gamma(h_{xy}d^{2\beta})\rangle.
\]
Cauchy--Schwarz for the semidefinite form induced by $\mathsf Q_\gamma$ gives
\[
        |M_{\alpha,\beta,\gamma}(x,y)|^2
        \le
        M_{\alpha,\alpha,\gamma}(x,y)
        M_{\beta,\beta,\gamma}(x,y).
\]
Write $x_{p+1}=x_1$.  Since $\mathsf M_{\alpha,\beta,\gamma}$ is positive semidefinite,
\begin{align*}
 \|\mathsf M_{\alpha,\beta,\gamma}\|_{S_p}^p
 &=\operatorname{tr}(\mathsf M_{\alpha,\beta,\gamma}^p)\\
 &\le \int\prod_{i=1}^p
       |M_{\alpha,\beta,\gamma}(x_i,x_{i+1})|\,d\mathbf x\\
 &\le \int\prod_{i=1}^p
       M_{\alpha,\alpha,\gamma}(x_i,x_{i+1})^{1/2}
       M_{\beta,\beta,\gamma}(x_i,x_{i+1})^{1/2}\,d\mathbf x\\
 &\le
   \left(\int\prod_{i=1}^pM_{\alpha,\alpha,\gamma}(x_i,x_{i+1})\,d\mathbf x\right)^{1/2}
   \left(\int\prod_{i=1}^pM_{\beta,\beta,\gamma}(x_i,x_{i+1})\,d\mathbf x\right)^{1/2}\\
 &=
   \operatorname{tr}(\mathsf M_{\alpha,\alpha,\gamma}^p)^{1/2}
   \operatorname{tr}(\mathsf M_{\beta,\beta,\gamma}^p)^{1/2}.
\end{align*}
This proves \eqref{eq:k2even-offdiagonal-M}.  Put $s_\alpha=1+\alpha+\gamma$ and $s_\beta=1+\beta+\gamma$.  Claim~\ref{clm:k2even-diagonal} gives
\[
        \|\mathsf M_{\alpha,\beta,\gamma}\|_{S_p}
        \le P^{2(s_\alpha+s_\beta)}=P^{2h}.
\]
\end{proof}

Taking square roots in Claim~\ref{clm:k2even-offdiagonal} proves \eqref{eq:k2even-half-bound} when $\sigma=\{a,b\}$, and the half-operator claim is complete.

\medskip
\noindent\emph{Step 3: assembly of an arbitrary subgraph.}

\begin{claim}[Assembly]\label{clm:k2even-assembly}
Every non-empty subgraph $S\subseteq\mathcal B_{2t}$ satisfies \eqref{eq:k2even-schatten-local}.
\end{claim}

\begin{proof}
Let $S\subseteq\mathcal B_{2t}$ be arbitrary, with the encoding fixed at the start of the proof.  By \eqref{eq:k2even-architecture-factorization},
\[
        \mathsf T_S^W
        =\mathsf X_{\sigma_\ell,\boldsymbol\tau^L}
         \mathsf X_{\sigma_r,\boldsymbol\tau^R}^*.
\]
If $h_L,h_R$ are the corresponding half-edge counts, then $h_L+h_R=e(S)$.  Schatten H\"older and \eqref{eq:k2even-half-bound} give
\[
\begin{aligned}
        \|\mathsf T_S^W\|_{S_p}
        &\le
        \|\mathsf X_{\sigma_\ell,\boldsymbol\tau^L}\|_{S_{2p}}
        \|\mathsf X_{\sigma_r,\boldsymbol\tau^R}\|_{S_{2p}}\\
        &\le P^{e(S)}
         =\|\mathsf T_{\mathcal B_{2t}}^W\|_{S_p}^{e(S)/(4t)}.
\end{aligned}
\]
This is \eqref{eq:k2even-schatten-local}.
\end{proof}

Claims~\ref{clm:k2even-endpoint-states}, \ref{clm:k2even-balanced}, \ref{clm:k2even-diagonal}, \ref{clm:k2even-offdiagonal} and \ref{clm:k2even-assembly} complete the proof.
\end{proof}

\subsection{Odd blocks and the parity classification}

\begin{proposition}[Odd complete-bipartite blocks fail at every cyclic length]\label{prop:birooted-k2odd-obstruction}
Let $m\ge3$ be odd.  For every integer $q\ge2$, the cyclic amalgam $C_q(\mathcal B_m)$ is not strongly dominating.  In particular, $\mathcal B_m$ is not a universal even-Schatten block.
\end{proposition}

\begin{proof}
Choose the two-valued function
\[
        f=\begin{cases}
        1,&\text{on a set of measure }\dfrac{2^m}{1+2^m},\\[2mm]
        -2,&\text{on its complement},
        \end{cases}
\]
so that
\[
        \int f^m=0,
        \qquad
        \int f^{m-1}>0.
\]
Let $W=f\otimes f$ and $c=\int f^2$.  A direct integration over the two vertices in the size-$2$ part gives
\begin{equation}\label{eq:k2odd-full-zero}
        T_{\mathcal B_m}^W(x,y)
        =f(x)^2f(y)^2c^{m-2}\left(\int f^m\right)^2=0.
\end{equation}
Let $S\subseteq\mathcal B_m$ be obtained by deleting both edges incident with one non-root vertex in the size-$m$ part.  Then
\begin{equation}\label{eq:k2odd-subgraph-positive}
        T_S^W(x,y)
        =f(x)^2f(y)^2c^{m-3}\left(\int f^{m-1}\right)^2,
\end{equation}
which is a non-zero positive semidefinite rank-one kernel.  The full cyclic density is zero by \eqref{eq:k2odd-full-zero}, whereas the subgraph obtained by placing $S$ in every block has density
\[
        \tr((\mathsf T_S^W)^q)>0.
\]
Thus $C_q(\mathcal B_m)$ fails to dominate this subgraph.
\end{proof}

\begin{corollary}[Parity classification of complete-bipartite cyclic sources]\label{cor:complete-bipartite-source-parity}
For every $m,q\ge2$,
\[
        C_q(\mathcal B_m)\text{ is strongly dominating}
        \quad\Longleftrightarrow\quad
        m\text{ is even}.
\]
Equivalently, $\mathcal B_m$ is a universal even-Schatten block if and only if $m$ is even.  When $m$ is even, the local inequality holds for every integer Schatten exponent $p\ge2$, and the full transfer operator is positive semidefinite.
\end{corollary}

\begin{proof}
For $m=2$, use Theorem~\ref{thm:birooted-k22-certificate}.  For even $m=2t\ge4$, use Theorem~\ref{thm:birooted-k2even-certificate}.  In both cases Theorem~\ref{thm:cyclic-schatten-criterion}(i) gives every cyclic length.  Proposition~\ref{prop:birooted-k2odd-obstruction} gives the converse and, in fact, excludes each cyclic length separately.
\end{proof}

\begin{proof}[Proof of Theorem~\ref{thm:necklace-family}]
The parity assertion is Corollary~\ref{cor:complete-bipartite-source-parity}.  The case $t=1$ is $C_q(\mathcal B_2)=R_q$, which is norming by Proposition~\ref{prop:Rn-reflection}.

Now let $t\ge2$ and write $N_{t,q}=C_q(\mathcal B_{2t})$.  It has
\[
        v(N_{t,q})=q(2t+1),
        \qquad
        e(N_{t,q})=4tq.
\]
The two vertices in the size-$2$ part of each block have degree $2t$, the shared roots have degree $4$, and the remaining vertices have degree $2$.  We check deletion by vertex type.  Deleting a private degree-$2$ vertex leaves its block connected between the two roots.  Deleting one of the two block centres leaves the other centre adjacent to every surviving vertex of that block, so the two roots remain connected through it.  Finally, deleting a shared root opens the cyclic sequence into a connected chain through the remaining roots; when $q=2$, the two surviving block interiors meet at the other shared root.  Thus deletion of any vertex leaves a connected graph, and $N_{t,q}$ is $2$-connected.

An edge from a block centre to a shared root has endpoint-degree pair $(2t,4)$, whereas an edge from a centre to a private vertex has pair $(2t,2)$.  These two edge types lie in different automorphism orbits, so $N_{t,q}$ is not edge-transitive.  The implication chain
\[
        \text{seminorming}\Longrightarrow\text{weakly norming}
        \Longrightarrow\text{edge-transitive}
\]
follows from \cite[Observation~2.5(i)]{Hatami} and \cite[Theorem~1]{SidorenkoEdgeTransitive}.  Therefore $N_{t,q}$ is not seminorming, and hence is not norming.  It contains a cycle, so it is not an even path.

Finally,
\[
        \frac{e(N_{t,q})}{v(N_{t,q})}=\frac{4t}{2t+1}
\]
is strictly increasing in $t$.  Thus an isomorphism determines $t$, and then the vertex count determines $q$.  The family $(N_{t,q})_{t,q\ge2}$ is pairwise non-isomorphic.
\end{proof}

\section{Structural obstructions and the outerplanar boundary}\label{sec:structural}

This section supplies the structural combinatorial half of the paper.  We first isolate four signed tests, then use them to classify trees, graphs with bridges, connected outerplanar graphs, and finally outerplanar cyclic sources.  The vertex-cover obstruction is the main tool: it converts a signed-kernel inequality into a rigid local degree condition.  The dependency of the classification on the witnesses is summarized below.

\begin{center}
\small
\renewcommand{\arraystretch}{1.12}
\begin{tabularx}{0.96\textwidth}{@{}l l Y@{}}
\toprule
forbidden structure & witness or lemma & consequence used later \\
\midrule
non-bipartiteness & Lemma~\ref{lem:bipartite} & every strongly dominating graph is bipartite \\
odd degree away from leaves & Lemmas~\ref{lem:rankone} and \ref{lem:odddegree} & parity restrictions on blocks and cut vertices \\
bridge with a cyclic side & Lemma~\ref{lem:leafcycle} and Theorem~\ref{thm:bridge} & reduction to trees, even paths and even stars \\
maximum-degree vertices fail to cover the edges & Theorem~\ref{thm:vertexcover} & explicit two-atom obstruction \\
degree-$2$ ear in a $2$-connected outerplanar block & Lemmas~\ref{lem:outerplanar-ear} and \ref{lem:two-terminal-outerplanar-ear} & exclusion of every non-cycle block except the stated source blocks \\
\bottomrule
\end{tabularx}
\end{center}

\subsection{Bipartiteness}

\begin{lemma}[Bipartiteness obstruction]\label{lem:bipartite}
If $H$ is strongly dominating, then $H$ is bipartite.
\end{lemma}

\begin{proof}
If $H$ is not bipartite, take a non-trivial complete bipartite graphon
\[
        W=1_{A\times B}+1_{B\times A}
\]
with $0<\mu(A),\mu(B)<1$.  Then $t_H(W)=0$, since every homomorphism from $H$ into the support graph of $W$ would give a bipartition of $H$.  But $t_{K_2}(W)>0$.  As $K_2\subseteq H$, this contradicts strong domination.
\end{proof}

\subsection{Rank-one degree moments}

\begin{lemma}[Rank-one factorization]\label{lem:rankone}
Let $G$ be any graph and let $W(x,y)=f(x)f(y)$ for a bounded real function $f$.  Then
\[
        t_G(W)=\prod_{v\in V(G)} \E f^{d_G(v)}.
\]
\end{lemma}

\begin{proof}
For a homomorphism variable $(x_v)_{v\in V(G)}$,
\[
        \prod_{uv\in E(G)} W(x_u,x_v)
        =
        \prod_{uv\in E(G)} f(x_u)f(x_v)
        =
        \prod_{v\in V(G)} f(x_v)^{d_G(v)}.
\]
Integrating over all vertices gives the formula.
\end{proof}

\begin{lemma}[Odd-degree obstruction]\label{lem:odddegree}
If $H$ is strongly dominating, then every vertex degree of $H$ is either $1$ or even.
\end{lemma}

\begin{proof}
Suppose some vertex of $H$ has odd degree $d\ge3$.  Choose $a>0$, $a\ne1$, and let
\[
        f=
        \begin{cases}
        1,&\text{with probability }\dfrac{a^d}{1+a^d},\\[2mm]
        -a,&\text{with probability }\dfrac1{1+a^d}.
        \end{cases}
\]
Then
\[
        \E f^d=0,
        \qquad
        \E f=\frac{a^d-a}{1+a^d}\ne0.
\]
For $W=f\otimes f$, Lemma \ref{lem:rankone} gives $t_H(W)=0$, while
\[
        t_{K_2}(W)=(\E f)^2>0.
\]
Since $K_2\subseteq H$, strong domination fails.
\end{proof}

\subsection{Zero-row-sum cancellation}

\begin{lemma}[Leaf-cycle obstruction]\label{lem:leafcycle}
If a graph $H$ has a leaf and contains a cycle, then $H$ is not strongly dominating.
\end{lemma}

\begin{proof}
Let $W=f\otimes f$, where $\E f=0$ and $\E f^2>0$.  If $u$ is a leaf of $H$ with neighbour $v$, then integrating first over $x_u$ gives
\[
        \int_0^1 W(x_u,x_v)\,dx_u=f(x_v)\E f=0.
\]
Therefore $t_H(W)=0$.

Let $C\subseteq H$ be a cycle.  Since every vertex of $C$ has degree $2$ within $C$, Lemma \ref{lem:rankone} gives
\[
        t_C(W)=(\E f^2)^{|V(C)|}>0.
\]
Thus $p_H(W)=0<p_C(W)$, contradicting $H\succeq_s C$.
\end{proof}

\subsection{An explicit additive certificate for the maximum-degree obstruction}

\begin{theorem}[Additive maximum-degree obstruction]\label{thm:vertexcover}
Let $H$ be a graph with at least one edge, maximum degree $\Delta$, and vertex-cover number $\tau(H)$.  If
\[
        \tau(H)>\frac{e(H)}{\Delta},
\]
then, for all sufficiently large $M$, a two-atom additive signed kernel $W_M$ satisfies
\[
        p_{K_{1,\Delta}}(W_M)>p_H(W_M).
\]
Consequently every strongly dominating graph satisfies
\[
        \tau(H)=\frac{e(H)}{\Delta(H)}.
\]
\end{theorem}

For connected hosts, the equality consequence is already implicit in Conlon--Lee \cite[Proposition~2.3]{ConlonLeeDomination}.  We prove the stronger certificate statement in Theorem \ref{thm:vertexcover}, valid without a connectivity hypothesis.  For a graph $G$, let $\tau(G)$ denote its vertex-cover number and $\Delta(G)$ its maximum degree.

For $M>1$, define a two-point random variable $\phi=\phi_M$ by
\[
        \phi=
        \begin{cases}
        M,&\text{with probability }\dfrac1{M^2+1},\\[2mm]
        -\dfrac1M,&\text{with probability }\dfrac{M^2}{M^2+1}.
        \end{cases}
\]
Then
\[
        \E\phi=0,
        \qquad
        \E\phi^2=1,
\]
and for every fixed integer $r\ge3$,
\[
        \E\phi^r
        =\frac{M^r+(-1)^rM^{2-r}}{M^2+1}
        =M^{r-2}+O(M^{r-4})
        \qquad(M\to\infty),
\]
with the implicit constant depending on $r$.

Let
\[
        W_M(x,y)=\phi(x)+\phi(y).
\]

\begin{lemma}[Additive expansion bound]\label{lem:additivebound}
For every graph $G$,
\[
        t_G(W_M)=O\left(M^{e(G)-2\tau(G)}\right)
\]
as $M\to\infty$.  If no surviving term exists in the expansion below, then $t_G(W_M)=0$, which is interpreted as satisfying the same bound.
\end{lemma}

\begin{proof}
Expand
\[
        \prod_{uv\in E(G)}(\phi(x_u)+\phi(x_v)).
\]
Each term corresponds to a choice of one endpoint $\omega(e)\in e$ for every edge $e$.  Let
\[
        a_\omega(v)=|\{e\in E(G):\omega(e)=v\}|
\]
be the load at $v$.  The contribution of $\omega$ is
\[
        \prod_{v\in V(G)} \E\phi^{a_\omega(v)}.
\]
This contribution vanishes if $a_\omega(v)=1$ for some $v$, since $\E\phi=0$.

If the contribution is non-zero, then every positive load is at least $2$.  The set
\[
        S_\omega=\{v:a_\omega(v)>0\}
\]
is a vertex cover of $G$, since each edge is assigned to one of its endpoints.  Therefore $|S_\omega|\ge \tau(G)$.  Also
\[
        \sum_{v\in S_\omega} a_\omega(v)=e(G).
\]
Using $\E\phi^0=1$, $\E\phi^2=1$, and $\E\phi^r=O(M^{r-2})$ for $r\ge3$, the absolute value of the contribution of $\omega$ is
\[
        O\left(M^{\sum_{v\in S_\omega}(a_\omega(v)-2)}\right)
        =O\left(M^{e(G)-2|S_\omega|}\right)
        \le O\left(M^{e(G)-2\tau(G)}\right).
\]
There are $2^{e(G)}$ choices of $\omega$, so summing over all choices proves the bound.
\end{proof}

\begin{proof}[Proof of Theorem \ref{thm:vertexcover}]
Let $\Delta=\Delta(H)$, and let $v$ be a vertex of degree $\Delta$.  Then $H$ contains $K_{1,\Delta}$.  If $\Delta=1$, the strict hypothesis $\tau(H)>e(H)/\Delta$ is impossible, so assume $\Delta\ge2$.  For the additive kernel $W_M$,
\[
        t_{K_{1,\Delta}}(W_M)=\E\phi^\Delta,
        \qquad
        |t_{K_{1,\Delta}}(W_M)|=\Theta(M^{\Delta-2}).
\]
Lemma~\ref{lem:additivebound} gives
\[
        |t_H(W_M)|=O(M^{e(H)-2\tau(H)}).
\]
If $\tau(H)>e(H)/\Delta$, then
\[
        \frac{e(H)-2\tau(H)}{e(H)}
        <
        \frac{\Delta-2}{\Delta}.
\]
Hence $p_H(W_M)<p_{K_{1,\Delta}}(W_M)$ for all sufficiently large $M$.  Thus $H$ is not strongly dominating.

Consequently, every strongly dominating graph satisfies $\tau(H)\le e(H)/\Delta(H)$.  The reverse inequality is the elementary covering bound $e(H)\le\Delta(H)\tau(H)$, so equality follows.
\end{proof}

\begin{corollary}[Equality structure]\label{cor:equalitycover}
Let $H$ be strongly dominating, and let $\Delta=\Delta(H)$.  Then there is a minimum vertex cover $S$ of $H$ such that every vertex in $S$ has degree $\Delta$, and every edge of $H$ has exactly one endpoint in $S$.
\end{corollary}

\begin{proof}
Let $S$ be any minimum vertex cover.  Since $S$ covers every edge,
\[
        e(H)\le \sum_{v\in S} d_H(v)\le \Delta |S|=\Delta\tau(H)=e(H).
\]
Thus both inequalities are equalities.  Hence every $v\in S$ has degree $\Delta$, and no edge is counted twice in $\sum_{v\in S}d_H(v)$.  The latter means that no edge has both endpoints in $S$, while the fact that $S$ is a vertex cover means that every edge has at least one endpoint in $S$.
\end{proof}

\begin{corollary}[Maximum-degree cover obstruction]\label{cor:maxdegreecover}
Let $H$ be strongly dominating, and let
\[
        V_\Delta(H)=\{v\in V(H):d_H(v)=\Delta(H)\}.
\]
Then $V_\Delta(H)$ is a vertex cover of $H$.  Equivalently, $H$ has no edge whose two endpoints both have degree smaller than $\Delta(H)$.
\end{corollary}

\begin{proof}
By Corollary \ref{cor:equalitycover}, $H$ has a vertex cover contained in $V_\Delta(H)$.  Hence the larger set $V_\Delta(H)$ is also a vertex cover.
\end{proof}

\begin{corollary}[Terminal-block obstruction]\label{cor:terminalblock}
Let $H$ be a connected graph with maximum degree $\Delta$.  Let $B$ be a terminal block of $H$ with cut vertex $x$.  If $B$ contains an edge $uv$ with $u,v\ne x$ and
\[
        d_H(u)<\Delta,\qquad d_H(v)<\Delta,
\]
then $H$ is not strongly dominating.
\end{corollary}

\begin{proof}
If $H$ were strongly dominating, Corollary \ref{cor:maxdegreecover} would force $V_\Delta(H)$ to be a vertex cover.  But the edge $uv$ has both endpoints outside $V_\Delta(H)$, a contradiction.
\end{proof}

\begin{remark}
Corollary \ref{cor:terminalblock} is the cut-vertex mechanism used in the outerplanar reductions below.  It does not claim that every cut vertex is impossible: terminal blocks with enough maximum-degree vertices can pass the cover test.  Thus the obstruction is a reduction tool, not a complete cut-vertex classification.  The cyclic constructions above show that this warning is essential.
\end{remark}

\subsection{Trees}

We now classify strongly dominating trees.  The positive examples $K_2$, even paths, and even stars have already been proved in Propositions~\ref{prop:evenpaths} and~\ref{prop:evenstars}.  The converse has four steps.

\begin{theorem}[Tree classification]\label{thm:trees}
Let $T$ be a tree with at least one edge.  Then $T$ is strongly dominating if and only if
\[
        T\cong K_2,
        \qquad
        T\cong P_{2s}\ (s\ge1),
        \qquad\text{or}\qquad
        T\cong K_{1,2s}\ (s\ge1).
\]
\end{theorem}

\paragraph{Odd internal degrees.}

If $T$ has an internal vertex of odd degree at least $3$, then Lemma \ref{lem:odddegree} already shows that $T$ is not strongly dominating.  Hence every strongly dominating tree has only even internal degrees.

\paragraph{Leaf-diameter rigidity.}

Assume now that all internal degrees of $T$ are even.  Let
\[
        m=e(T),\qquad L=\#\{\text{leaves of }T\},\qquad D=\diam(T).
\]

\begin{lemma}[Rademacher leaf obstruction]\label{lem:rademacher}
If $T$ is strongly dominating and all its internal degrees are even, then
\[
        \frac{L}{m}\le \frac2D.
\]
\end{lemma}

\begin{proof}
Let
\[
        f=
        \begin{cases}
        1,&\text{with probability }(1+\varepsilon)/2,\\
        -1,&\text{with probability }(1-\varepsilon)/2,
        \end{cases}
\]
where $0<\varepsilon<1$.  Then
\[
        \E f=\varepsilon,
        \qquad
        \E f^{2j}=1\quad(j\ge1).
\]
For $W=f\otimes f$, Lemma \ref{lem:rankone} gives
\[
        t_T(W)=\varepsilon^L,
\]
because only leaves contribute odd moments.  Let $P_D\subseteq T$ be a diameter path.  Then
\[
        t_{P_D}(W)=\varepsilon^2.
\]
Since $T\succeq_s P_D$, we must have
\[
        \varepsilon^{L/m}\ge \varepsilon^{2/D}
\]
for every $0<\varepsilon<1$, which implies $L/m\le 2/D$.
\end{proof}

\begin{lemma}[Leaf-diameter inequality]\label{lem:leafdiam}
Every tree $T$ satisfies
\[
        e(T)\le \frac{L(T)\diam(T)}2.
\]
Equality holds if and only if $T$ is a path or an equal-arm spider.
\end{lemma}

\begin{proof}
Let $c$ be the metric center of $T$, allowing $c$ to be either a vertex or the midpoint of an edge.  Every leaf $\ell$ satisfies
\[
        d(c,\ell)\le \frac{\diam(T)}2.
\]
Therefore
\[
        \sum_{\ell\text{ leaf}} d(c,\ell)
        \le
        \frac{L(T)\diam(T)}2.
\]
On the other hand, every edge of $T$ lies on at least one path from $c$ to a leaf.  Hence
\[
        e(T)\le \sum_{\ell\text{ leaf}} d(c,\ell),
\]
which proves the inequality.

If equality holds, then every leaf has distance exactly $\diam(T)/2$ from $c$, and every edge is counted by exactly one path from $c$ to a leaf.  The second condition forbids branching away from $c$: every component of $T-c$ contains exactly one leaf.  Thus $T$ is a path if the center is an edge midpoint or if the center vertex has degree $2$, and otherwise $T$ is a spider whose arms all have equal length.

Conversely, paths and equal-arm spiders plainly attain equality.
\end{proof}

Now let $T$ be strongly dominating and suppose all internal degrees are even.  The leaf-diameter inequality gives $L/m\ge 2/D$, while Lemma \ref{lem:rademacher} gives $L/m\le 2/D$.  Equality must therefore hold in Lemma \ref{lem:leafdiam}, so $T$ is a path or an equal-arm spider.

\paragraph{Odd paths.}

\begin{lemma}[Odd paths]\label{lem:oddpaths}
If $m\ge3$ is odd, then $P_m$ is not strongly dominating.
\end{lemma}

\begin{proof}
Choose $a>0$, $a\ne1$, and set
\[
        w=\frac{a^m}{1+a^m}.
\]
We construct a bounded self-adjoint finite-rank kernel whose associated operator $A$ has eigenvalues $1$ and $-a$, with spectral weights $w$ and $1-w$ relative to the vector $\one$.

Let $h$ be a bounded mean-zero function with $\|h\|_2=1$, for instance the Rademacher function taking values $1$ and $-1$ on two intervals of measure $1/2$.  Define
\[
        u=\sqrt w\,\one+\sqrt{1-w}\,h,
        \qquad
        v=\sqrt{1-w}\,\one-\sqrt w\,h.
\]
Then $u,v$ are orthonormal and
\[
        \one=\sqrt w\,u+\sqrt{1-w}\,v.
\]
Let
\[
        A=P_u-aP_v,
\]
where $P_u$ and $P_v$ are the rank-one orthogonal projections onto $u$ and $v$.  This operator is represented by a bounded symmetric kernel.

For this kernel,
\[
        t_{P_m}(W)=\langle \one,A^m\one\rangle
        =w+(1-w)(-a)^m.
\]
Since $m$ is odd,
\[
        t_{P_m}(W)=w-(1-w)a^m=0.
\]
But
\[
        t_{K_2}(W)=\langle \one,A\one\rangle
        =w-(1-w)a
        =\frac{a^m-a}{1+a^m}\ne0.
\]
Thus $P_m\not\succeq_s K_2$.
\end{proof}

\paragraph{Equal-arm spiders.}

Let $S_{r,s}$ denote the spider with $r$ arms, each of length $s$.  Thus $S_{2,s}=P_{2s}$ and $S_{r,1}=K_{1,r}$.

\begin{lemma}[Non-star spiders]\label{lem:spiders}
Let $r\ge4$ be even and $s\ge2$.  Then $S_{r,s}$ is not strongly dominating.
\end{lemma}

\begin{proof}
Use the two-point random variable $\phi=\phi_M$ from the additive maximum-degree obstruction and the kernel
\[
        W(x,y)=\phi(x)+\phi(y).
\]
Let $A=A_W$.  Since $\E\phi=0$ and $\E\phi^2=1$,
\[
        A\one=\phi,
        \qquad
        A\phi=\one.
\]
Therefore
\[
        A^s\one=
        \begin{cases}
        \one,&s\text{ even},\\
        \phi,&s\text{ odd}.
        \end{cases}
\]
Rooting the spider at its center gives
\[
        t_{S_{r,s}}(W)=\int (A^s\one(x))^r\,dx.
\]
The spider contains $K_{1,3}$ as a subgraph, and
\[
        t_{K_{1,3}}(W)=\int (A\one)^3=\E\phi^3\sim M.
\]
If $s$ is even, then $t_{S_{r,s}}(W)=1$, while $p_{K_{1,3}}(W)\sim M^{1/3}$, so domination fails.

If $s$ is odd and $s\ge3$, then
\[
        t_{S_{r,s}}(W)=\E\phi^r\sim M^{r-2}.
\]
Consequently
\[
        p_{S_{r,s}}(W)\sim M^{(r-2)/(rs)},
        \qquad
        p_{K_{1,3}}(W)\sim M^{1/3}.
\]
Since $r\ge4$ and $s\ge3$,
\[
        \frac{r-2}{rs}<\frac13.
\]
Thus $p_{K_{1,3}}(W)>p_{S_{r,s}}(W)$ for all sufficiently large $M$.
\end{proof}

\begin{proof}[Proof of Theorem \ref{thm:trees}]
The positive cases were proved in Propositions~\ref{prop:evenpaths} and~\ref{prop:evenstars}.  Conversely, let $T$ be strongly dominating.  By Lemma \ref{lem:odddegree}, all internal degrees of $T$ are even.  The equality argument following Lemma \ref{lem:leafdiam} shows that $T$ is a path or an equal-arm spider.  Lemma \ref{lem:oddpaths} eliminates odd paths of length at least $3$.  Lemma \ref{lem:spiders} eliminates equal-arm spiders with at least four arms and arm length at least two.  Odd stars are eliminated by Lemma \ref{lem:odddegree}.  The only remaining trees are $K_2$, even paths, and even stars.
\end{proof}

\subsection{Graphs with bridges}

\begin{theorem}[Bridge theorem]\label{thm:bridge}
Let $H$ be a connected graph with at least one edge and a bridge.  Then $H$ is strongly dominating if and only if
\[
        H\cong K_2,
        \qquad H\cong P_{2s},
        \qquad\text{or}\qquad H\cong K_{1,2s}
\]
for some $s\ge1$.
\end{theorem}

\begin{lemma}[A bridge forces leaves]\label{lem:bridgeleaves}
Let $H$ be a connected strongly dominating graph.  If $H$ has a bridge, then $H$ has a leaf on each side of that bridge.
\end{lemma}

\begin{proof}
Let $e=uv$ be a bridge.  Removing $e$ gives two components with vertex sets $A$ and $B$.  In $H$,
\[
        \sum_{x\in A} d_H(x)=2e(H[A])+1,
\]
because exactly one edge leaves $A$.  Thus the degree sum over $A$ is odd, so $A$ contains a vertex of odd degree in $H$.  By Lemma \ref{lem:odddegree}, every odd degree in $H$ is $1$.  Hence $A$ contains a leaf.  The same argument applies to $B$.
\end{proof}

\begin{lemma}[Bridge parity under leaf-only odd degrees]\label{lem:bridgeparity}
Let $H$ be a connected graph in which every odd-degree vertex is a leaf.  If $e$ is a bridge of $H$, then each component of $H-e$ contains a leaf of $H$.
\end{lemma}

\begin{proof}
Let $A$ and $B$ be the vertex sets of the two components of $H-e$.  Since exactly one edge leaves $A$,
\[
        \sum_{x\in A} d_H(x)=2e(H[A])+1,
\]
so $A$ contains a vertex of odd degree in $H$.  By hypothesis this vertex is a leaf.  The same argument applies to $B$.
\end{proof}

\begin{proof}[Proof of Theorem \ref{thm:bridge}]
The graphs $K_2$, $P_{2s}$, and $K_{1,2s}$ are strongly dominating by Propositions~\ref{prop:evenpaths} and~\ref{prop:evenstars}.  Conversely, let $H$ be connected, strongly dominating, and suppose $H$ has a bridge.  By Lemma \ref{lem:bridgeleaves}, $H$ has a leaf.  If $H$ contained a cycle, Lemma \ref{lem:leafcycle} would show that $H$ is not strongly dominating.  Therefore $H$ is a tree.  Theorem \ref{thm:trees} completes the proof.
\end{proof}

\begin{corollary}[Core reduction]\label{cor:core}
Every connected strongly dominating graph is either
\[
        K_2,
        \qquad P_{2s},
        \qquad K_{1,2s},
\]
or it is bridgeless, bipartite, and all its vertex degrees are even.
\end{corollary}

\begin{proof}
By Lemma \ref{lem:bipartite}, the graph is bipartite.  By Lemma \ref{lem:odddegree}, all vertex degrees are $1$ or even.  If the graph has a bridge, Theorem \ref{thm:bridge} applies.  If it has no bridge, then it has no leaf unless it is $K_2$, and therefore all degrees are even.
\end{proof}

\subsection{Outerplanar graphs}\label{sec:outerplanar}

The bridge and vertex-cover obstructions extend to all outerplanar graphs once the terminal-cycle argument is replaced by an outer-face ear.

\begin{lemma}[Outerplanar degree-$2$ ear]\label{lem:outerplanar-ear}
Let $B$ be a $2$-connected bipartite outerplanar graph which is not a cycle.  For every prescribed vertex $x\in V(B)$, there is an edge $uv\in E(B)$ such that
\[
        u,v\ne x,
        \qquad
        d_B(u)=d_B(v)=2.
\]
\end{lemma}

\begin{proof}
Fix an outerplane embedding of $B$.  A $2$-connected outerplanar graph has a Hamilton cycle bounding the outer face, and the weak dual of its bounded faces is a tree; see, for example, \cite{Diestel}.  Since $B$ is not a cycle, it has a chord, so the weak dual has at least two leaves.

A leaf face is bounded by one chord and an outer-boundary path $P$.  Every internal vertex of $P$ has degree $2$ in $B$, since any further chord at such a vertex would subdivide the face.  The facial cycle is even because $B$ is bipartite.  Hence $P$ has odd length.  It has length at least $3$, because the chord is not an outer edge, so $P$ contains an edge whose two endpoints are internal vertices and therefore have degree $2$.

A vertex internal to such a path has degree $2$ and is incident with only one bounded face, so the prescribed vertex $x$ can be internal to at most one leaf-face path.  Choose another leaf face.  An edge between two internal vertices of its path avoids $x$ and has the required degrees.
\end{proof}

\begin{proof}[Proof of Theorem \ref{thm:outerplanar}]
The four displayed families are strongly dominating by Propositions \ref{prop:evenpaths}, \ref{prop:evenstars}, and \ref{prop:evencycles}.

Conversely, let $H$ be a connected strongly dominating outerplanar graph.  If $H$ has a bridge, Theorem \ref{thm:bridge} gives $K_2$, an even path, or an even star.  We may therefore assume that $H$ is bridgeless.  By Corollary \ref{cor:core}, $H$ is bipartite and every vertex has even degree.

Suppose first that $H$ is $2$-connected.  If it is not a cycle, Lemma \ref{lem:outerplanar-ear} gives an edge $uv$ with $d_H(u)=d_H(v)=2$.  Since $H$ has a chord, some vertex has degree at least $3$ and hence, by evenness, $\Delta(H)\ge4$.  Thus neither endpoint of $uv$ has maximum degree, contradicting Corollary \ref{cor:maxdegreecover}.  Hence $H$ is a cycle, and bipartiteness makes it an even cycle.

It remains to exclude a bridgeless graph with a cut vertex.  Choose a terminal block $B$ with cut vertex $x$.  Since $H$ has no bridges, $B$ is $2$-connected.  If $B$ is a cycle, it has an edge $uv$ avoiding $x$ whose endpoints have degree $2$ in $B$.  If $B$ is not a cycle, Lemma \ref{lem:outerplanar-ear} gives such an edge.  Because $B$ is terminal, $u$ and $v$ have no neighbours outside $B$, so
\[
        d_H(u)=d_H(v)=2.
\]
The cut vertex $x$ belongs to at least two non-trivial blocks, each contributing at least two incident edges, and therefore $\Delta(H)\ge d_H(x)\ge4$.  Corollary \ref{cor:terminalblock} now excludes $H$.  This completes the classification.
\end{proof}

\subsubsection{Two-atom witnesses}

The arguments above are constructive.  Throughout the outerplanar classification, failure of strong domination is detected by a two-atom kernel.  We record this as a final theorem.

\begin{theorem}[Two-atom witness theorem]\label{thm:witness}
Let $H$ be a connected outerplanar graph with at least one edge.  If $H$ is not one of
\[
        K_2,
        \qquad P_{2s},
        \qquad K_{1,2s},
        \qquad C_{2s},
\]
then there exists a non-empty subgraph $F\subseteq H$ and a bounded symmetric two-atom real kernel $W$ such that
\[
        p_F(W)>p_H(W).
\]
Moreover, one may choose $F$ among the following types:
\[
        K_2,
        \qquad C_\ell,
        \qquad P_D,
        \qquad K_{1,3},
        \qquad K_{1,\Delta(H)}.
\]
\end{theorem}

\begin{proof}
We trace through the proof of the classification.

If $H$ is non-bipartite, a two-step complete bipartite graphon gives $t_H(W)=0$ and $t_{K_2}(W)>0$.

If $H$ has an odd-degree vertex of degree at least $3$, the two-point rank-one kernel from Lemma \ref{lem:odddegree} gives $t_H(W)=0$ and $t_{K_2}(W)>0$.

If $H$ has both a bridge and a cycle, then either the preceding odd-degree case applies, or every odd-degree vertex of $H$ is a leaf.  In the latter case, Lemma \ref{lem:bridgeparity} gives a leaf of $H$, and Lemma \ref{lem:leafcycle} gives a rank-one two-point kernel with $t_H(W)=0$ and $t_C(W)>0$ for some cycle $C\subseteq H$.

If $H$ is a tree not surviving the classification, the proof of Theorem \ref{thm:trees} gives one of the following witnesses: a rank-one moment witness against $K_2$ for an odd internal degree; a biased Rademacher rank-one witness against a diameter path $P_D$; the two-eigenvalue path witness against $K_2$ for an odd path; or the additive two-point witness against $K_{1,3}$ for a non-star equal-arm spider.

Finally, suppose that $H$ is bridgeless and is not an even cycle.  The proof of Theorem \ref{thm:outerplanar} produces an edge whose two endpoints have degree $2<\Delta(H)$.  Hence the maximum-degree vertices do not form a vertex cover.  Equivalently,
\[
        \tau(H)>\frac{e(H)}{\Delta(H)}.
\]
The additive two-point kernel in the proof of Theorem \ref{thm:vertexcover}, compared with the maximum-degree star $K_{1,\Delta(H)}\subseteq H$, yields
\[
        p_{K_{1,\Delta(H)}}(W)>p_H(W)
\]
for all sufficiently large values of the parameter $M$.

In all cases the witness is supported on two atoms.  Scaling puts the kernel in any prescribed bounded interval without changing the strict normalized violation.
\end{proof}

\begin{remark}
The two-atom statement records a finite-dimensional form of every obstruction used above.  Step kernels are the finite weighted models underlying graphon approximation and dense graph limits \cite{LovaszSzegedyLimits,BorgsChayesLovaszSosVesztergombi,LovaszBook}.  Thus, throughout the outerplanar class the obstruction to strong domination is finite-dimensional: no genuinely infinite-dimensional graphon or operator is needed to witness failure.  The signed cancellations already appear in the smallest possible probability spaces.
\end{remark}

\subsection{Outerplanar cyclic sources}\label{sec:outerplanar-sources}

Theorem~\ref{thm:exact-cyclic-source} turns the source problem into a local spectral condition.  In the outerplanar class that condition has a finite combinatorial description.  The key is a two-terminal strengthening of Lemma~\ref{lem:outerplanar-ear}.

\begin{lemma}[Two-terminal outerplanar ear]\label{lem:two-terminal-outerplanar-ear}
Let $Q$ be a $2$-connected bipartite outerplanar graph which is not a cycle, and let $X\subseteq V(Q)$ with $|X|\le2$.  If every vertex in $V(Q)\setminus X$ has even degree in $Q$, then there is an edge $uv\in E(Q)$ such that
\[
        u,v\notin X,
        \qquad
        d_Q(u)=d_Q(v)=2.
\]
\end{lemma}

\begin{figure}[t]
\centering
\begin{tikzpicture}[scale=0.88, every node/.style={font=\small}]
  \begin{scope}[xshift=-3.7cm]
    \node[circle,draw,inner sep=2pt] (f1) at (0,0) {$F$};
    \node[circle,draw,inner sep=2pt] (f2) at (1.2,0) {$F'$};
    \node (dots) at (2.35,0) {$\cdots$};
    \node[circle,draw,inner sep=2pt] (fk) at (3.5,0) {$G$};
    \draw (f1)--(f2)--(dots)--(fk);
    \node at (1.75,-0.7) {weak dual $T$};
    \node[font=\scriptsize,align=center] at (0,0.75) {leaf\\face};
    \node[font=\scriptsize,align=center] at (3.5,0.75) {leaf\\face};
    \node at (0,-1.25) {$x_F\in X$};
    \node at (3.5,-1.25) {$x_G\in X$};
  \end{scope}

  \begin{scope}[xshift=2.65cm]
    \coordinate (a) at (0,0);
    \coordinate (u) at (0.65,0.9);
    \coordinate (v) at (1.55,0.9);
    \coordinate (b) at (2.2,0);
    \coordinate (c) at (0.35,-1.0);
    \coordinate (d) at (1.85,-1.0);
    \draw[thick] (a)--(u)--(v)--(b);
    \draw[thick] (a)--(c);
    \draw[thick] (b)--(d);
    \draw[very thick] (a)--(b);
    \draw[dashed,thick] (c)--(d);
    \fill (a) circle (1.5pt) node[left] {$a$};
    \fill (b) circle (1.5pt) node[right] {$b$};
    \fill (u) circle (1.3pt) node[above] {$u$};
    \fill (v) circle (1.3pt) node[above] {$v$};
    \fill (c) circle (1.3pt);
    \fill (d) circle (1.3pt);
    \node at (1.1,0.47) {$F$};
    \node at (1.1,-0.48) {$F'$};
    \node[font=\scriptsize,fill=white,inner sep=1pt] at (1.1,0.12) {chord $ab$};
    \node[font=\scriptsize,fill=white,inner sep=1pt] at (1.1,-1.18) {possible second chord $e_c$};
    \node[font=\scriptsize] at (1.1,-1.62) {(b) an end face and its neighbour};
  \end{scope}
\end{tikzpicture}
\caption{The terminal configuration in Lemma~\ref{lem:two-terminal-outerplanar-ear}.  If the desired degree-$2$ edge is absent, the weak dual is forced to be a path with a $4$-cycle at each end.  At an end chord $ab$, one endpoint is not incident with the possible second chord of the neighbouring face.  }
\label{fig:two-terminal-ear}
\end{figure}
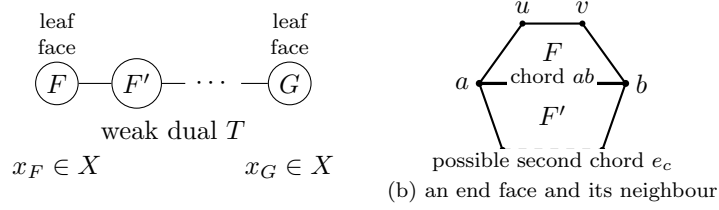

\begin{proof}
Fix an outerplane embedding.  Let $C$ be the Hamilton cycle bounding the outer face, and let $T$ be the weak dual whose vertices are the bounded faces and whose edges correspond to shared chords.  Since $Q$ is not a cycle, it has a chord, so $T$ is a non-trivial tree.

\medskip
\noindent\emph{Leaf faces.}
Let $F$ be a leaf of $T$.  Exactly one edge of $F$ is a chord; call it $ab$.  The remainder of the facial boundary is an $a$--$b$ path $P_F$ contained in $C$.  Every internal vertex of $P_F$ has degree $2$ in $Q$: a further chord at such a vertex would lie inside the region bounded by $P_F\cup ab$ and would split the face $F$.  Since $Q$ is bipartite, the facial cycle is even, and hence $P_F$ has odd length at least $3$.  If $I_F$ denotes the set of internal vertices of $P_F$, then $|I_F|$ is even and at least $2$, and the vertices in $I_F$ induce a path.

Suppose, for a contradiction, that there is no edge outside $X$ whose endpoints both have degree $2$.  Then $X\cap I_F$ is a vertex cover of the path induced by $I_F$ for every leaf face $F$.  A vertex internal to $P_F$ has degree $2$ and is incident with no other bounded face, so the sets $I_F$ belonging to distinct leaf faces are disjoint.  Thus distinct leaves of $T$ require distinct vertices of $X$.

Every non-trivial tree has at least two leaves.  Since $|X|\le2$, it follows that $T$ has exactly two leaves, $|X|=2$, and each of the two leaf paths contains exactly one vertex of $X$.  A path on $k=|I_F|$ vertices has vertex-cover number $k/2$ because $k$ is even.  One vertex covers it only when $k=2$.  Hence each leaf face is a $4$-cycle, with outer path
\[
        a-u-v-b,
\]
where exactly one of $u,v$ belongs to $X$.  A tree with exactly two leaves is a path, so the weak dual $T$ is a path, as illustrated in Figure~\ref{fig:two-terminal-ear}.

\medskip
\noindent\emph{The end chord.}

\begin{claim}[An endpoint of the end chord has degree $3$]\label{clm:outer-ear-endpoint}
Let $F$ be an end face, let $ab$ be its chord, and let $F'$ be the unique neighbouring bounded face.  At least one endpoint of $ab$ is outside $X$, is incident with no chord other than $ab$, and consequently has degree $3$ in $Q$.
\end{claim}

\begin{proof}
In fact both $a$ and $b$ lie outside $X$.  The two elements of $X$ lie in the internal vertex sets of the two opposite leaf paths, while an internal leaf-path vertex has degree $2$ and is incident with no chord; hence it cannot be an endpoint of $ab$.

Since the weak dual is a path, the vertex corresponding to $F'$ has degree at most $2$.  Thus the boundary of $F'$ contains at most one chord other than $ab$; call it $e_c$ if it exists, as in Figure~\ref{fig:two-terminal-ear}.  The chord $e_c$ can be incident with at most one of $a,b$, so choose an endpoint, say $a$, which is not incident with $e_c$.

Inspect the cyclic order of edges around $a$.  On the $F$ side, the face boundary uses $ab$ and one outer-cycle edge at $a$.  On the $F'$ side, because $a$ is not incident with $e_c$, the face boundary uses $ab$ and the other outer-cycle edge.  Any further chord at $a$ would lie between $ab$ and one of these outer edges and would create another bounded face adjacent to $F$ or $F'$.  This contradicts, respectively, that $F$ is a leaf face or that $F'$ has no second chord at $a$.  Hence $ab$ is the only chord incident with $a$, and $a$ has exactly two outer-cycle edges and $ab$.  Therefore $d_Q(a)=3$.
\end{proof}

Claim~\ref{clm:outer-ear-endpoint} gives a vertex $a\notin X$ of odd degree, contradicting the hypothesis.  This proves the lemma.
\end{proof}

For $k\ge1$, let $\mathcal D_k$ be the birooted graph obtained from copies $Q_1,\ldots,Q_k$ of $C_4$, each with distinguished opposite vertices $a_i,b_i$, by identifying $b_i$ with $a_{i+1}$ for $1\le i<k$ and taking $a_1,b_k$ as the roots.  Equivalently, $\mathcal D_k$ is the $K_{2,2}$-replacement of the path $P_k$.

\begin{proof}[Proof of Theorem~\ref{thm:outerplanar-source}]
Assume first that $B$ is universal, and put
\[
        H=C_2(B).
\]
By Theorem~\ref{thm:exact-cyclic-source}, $H$ is strongly dominating.

\begin{claim}[The doubled block is bridgeless]\label{clm:outer-source-bridgeless}
The graph $H$ has no bridge.
\end{claim}

\begin{proof}
Since $B$ is connected, choose an $\ell$--$r$ path in each of the two copies.  Their union is a cycle in $H$.  If $H$ had a bridge, Theorem~\ref{thm:bridge} would force $H$ to be $K_2$, an even path, or an even star, none of which contains a cycle.  Hence $H$ is bridgeless.
\end{proof}

If $\Delta(H)=2$, then Claim~\ref{clm:outer-source-bridgeless} and connectivity imply that every vertex of $H$ has degree $2$, so $H$ is a cycle.  Every non-root vertex of $B$ consequently has degree $2$, while each shared vertex of $H$ has degree
\[
        d_B(\ell)+d_B(r)=2.
\]
Connectivity gives $d_B(\ell),d_B(r)\ge1$, so both root degrees are $1$.  Thus $B$ is a path with the roots at its endpoints.  The case of one edge is excluded by the standing assumption that the roots are non-adjacent, so $B\cong P_k$ with $k\ge2$.

We may henceforth assume that $\Delta(H)\ge4$.

\begin{claim}[The block-cut tree is a path between the roots]\label{clm:outer-source-block-path}
If $B$ has more than one block, then every leaf block contains exactly one root, that root is not a cut vertex, and the block-cut tree of $B$ is a path whose two leaf blocks contain $\ell$ and $r$, respectively.
\end{claim}

\begin{proof}
Use the standard bipartite block-cut tree: its nodes are the blocks and the cut vertices of $B$, and a block node is adjacent to the cut-vertex nodes it contains.  We regard a bridge as a block isomorphic to $K_2$.  Every cut-vertex node has degree at least $2$, so every leaf of the block-cut tree is a block node.  First consider a leaf block avoiding both roots.  If it is a bridge, its non-cut endpoint is a non-root vertex of degree $1$ in $H$, and the incident edge is a bridge of $H$, contradicting Claim~\ref{clm:outer-source-bridgeless}.  If it is a $2$-connected block $Q$ meeting the rest of $B$ only at the cut vertex $x$, every vertex of $Q-x$ has the same degree in $Q$, $B$, and $H$, and this degree is even by Lemma~\ref{lem:odddegree}.  If $Q$ is not a cycle, Lemma~\ref{lem:two-terminal-outerplanar-ear} with $X=\{x\}$ gives an edge whose endpoints have degree $2$ in $H$; if $Q$ is a cycle, choose an edge avoiding $x$.  In either case neither endpoint has degree $\Delta(H)$, contradicting Corollary~\ref{cor:maxdegreecover}.  Thus every leaf block contains a root.

A leaf block whose unique cut vertex is a root and which contains no other root is also impossible.  For a leaf bridge the non-root endpoint again yields a bridge of $H$.  For a $2$-connected leaf block the preceding argument applies with $X$ equal to that root.

No leaf block can contain both roots when the block-cut tree is non-trivial.  If neither root were its unique cut vertex, another leaf block would avoid both roots.  If one root were the unique cut vertex, another leaf block would have to contain that same root and no other root, which is the configuration just excluded.  A leaf block cannot have both roots as cut vertices because it has only one cut vertex.  Hence every leaf block contains exactly one root, and that root is not a cut vertex.  Distinct leaf blocks therefore require distinct roots.  A non-trivial finite tree has at least two leaves, and we have shown that its leaf nodes are precisely two block nodes, one containing each root.  A finite tree with exactly two leaves is a path, so the block-cut tree is a path between those two leaf blocks.
\end{proof}

In the one-block case assign the two roots as its ports.  In the multi-block case, Claim~\ref{clm:outer-source-block-path} orders the blocks along a path; give a terminal block its root and its unique cut vertex as ports, and give an internal block its two cut vertices as ports.

\begin{claim}[The only $2$-connected blocks are opposite-rooted $4$-cycles]\label{clm:outer-source-c4-blocks}
Every $2$-connected block $Q$ of $B$ is a copy of $C_4$, and its two ports are opposite vertices.
\end{claim}

\begin{proof}
Every vertex outside the two ports is neither a root nor a cut vertex, so its degree in $Q$ equals its degree in $H$ and is even.  If $Q$ is not a cycle, Lemma~\ref{lem:two-terminal-outerplanar-ear} produces an edge outside the ports whose endpoints have degree $2$ in $H$, contradicting Corollary~\ref{cor:maxdegreecover}.  Hence $Q$ is a cycle.

An edge of this cycle with both endpoints outside the ports would give the same contradiction.  Thus the two ports form a vertex cover of the cycle.  A cycle with a vertex cover of size $2$ has length at most $4$; since $Q$ is bipartite and simple, it is $C_4$, and the two covering vertices are opposite.
\end{proof}

\begin{claim}[Mixed block types are impossible]\label{clm:outer-source-uniform-blocks}
All blocks of $B$ are edges, or all blocks are opposite-rooted copies of $C_4$.
\end{claim}

\begin{proof}
By Claim~\ref{clm:outer-source-c4-blocks}, every block is either a single edge or an opposite-rooted $C_4$.  At a cut vertex between consecutive blocks, an edge block contributes $1$ to the degree and a $C_4$ block contributes $2$.  A mixed pair would therefore give a non-root cut vertex of degree $3$ in $H$.  This contradicts Lemma~\ref{lem:odddegree}, which permits odd degree only at leaves.  Hence adjacent blocks have the same type, and the block path is uniform.
\end{proof}

Claims~\ref{clm:outer-source-block-path}--\ref{clm:outer-source-uniform-blocks} show that $B$ is either a path rooted at its endpoints or a chain $\mathcal D_k$ of opposite-rooted $4$-cycles.  As noted above, the non-adjacent-root hypothesis gives $k\ge2$ in the path case.

Conversely, if $B=P_k$ with $k\ge2$, then for every even $q$,
\[
        C_q(B)\cong C_{kq},
\]
and this is an even cycle, hence norming and strongly dominating.  Theorem~\ref{thm:exact-cyclic-source} shows that $B$ is universal.

If $B=\mathcal D_k$, then
\[
        C_q(B)\cong R_{kq}.
\]
Proposition~\ref{prop:Rn-reflection} says that this graph is norming for every $q\ge2$.  The cyclic transfer criterion again proves universality and completes the classification.
\end{proof}

The theorem gives the promised finite interpretation of the spectral source condition in a non-trivial minor-closed class.  It also identifies exactly where the present non-norming construction must leave that class.

\begin{corollary}[$K_{2,3}$ threshold for root-reversible series-parallel blocks]\label{cor:k23-threshold}
Let $(B,\ell,r)$ be a connected root-reversible series-parallel universal even-Schatten block with non-adjacent roots.  If $C_q(B)$ is not seminorming for some even $q\ge2$, then $B$ contains a $K_{2,3}$ minor.
\end{corollary}

\begin{proof}
A series-parallel graph has no $K_4$ minor.  If $B$ also had no $K_{2,3}$ minor, the forbidden-minor characterization of outerplanar graphs would make $B$ outerplanar; see, for example, \cite{Diestel}.  Theorem~\ref{thm:outerplanar-source} would then make every even $C_q(B)$ norming, a contradiction.
\end{proof}

For every $t\ge2$, the block $\mathcal B_{2t}$ contains a $K_{2,3}$ minor but no $K_4$ minor, so Corollary~\ref{cor:k23-threshold} is sharp at the level of excluded minors and is witnessed by infinitely many universal blocks.  Their cyclic amalgams then acquire a $K_4$ minor by Proposition~\ref{prop:cyclic-k4-minor} below.  Thus the local sources first cross the outerplanar boundary, while the global families subsequently cross the series-parallel boundary.

\begin{proposition}[A minor obstruction for cyclic sources]\label{prop:cyclic-k4-minor}
Let $(B,\ell,r)$ be a connected birooted graph.  Suppose that $B-\{\ell,r\}$ contains disjoint connected vertex sets $A$ and $C$ such that each of $A$ and $C$ has a neighbour at both roots and there is an edge between $A$ and $C$.  Then, for every $q\ge2$, the cyclic amalgam $C_q(B)$ contains a $K_4$ minor.  In particular, no member of this cyclic family is series-parallel.
\end{proposition}

\begin{proof}
Choose one copy of $B$ in $C_q(B)$ and denote its two junction vertices by $\ell$ and $r$.  Because $B$ is connected, every other copy contains a path between its two roots.  Concatenating such paths through the remaining $q-1$ copies gives an $r$--$\ell$ path $P$ whose internal vertices avoid the chosen copy.

Contract $A$ and $C$ separately.  Contract all of $P$ except its terminal vertex $\ell$ to a connected branch set $D$ containing $r$.  The four disjoint branch sets
\[
        \{\ell\},\qquad A,\qquad C,\qquad D
\]
are pairwise adjacent.  The sets $A$ and $C$ are adjacent by hypothesis; each is adjacent to both $\{\ell\}$ and $D$ through its prescribed neighbours at the two roots; and $\{\ell\}$ is adjacent to $D$ through the final edge of $P$.  They therefore form a $K_4$ minor.
\end{proof}

The hypothesis holds for every $\mathcal B_{2t}$ with $t\ge2$: take $A=\{a\}$ and let $C$ consist of $b$ together with one private vertex.  Thus the escape from the series-parallel class follows from a local two-channel structure shared by the entire complete-bipartite source family.

\section{Further questions}\label{sec:further-directions}

Theorem~\ref{thm:outerplanar-source} settles the cyclic source problem for root-reversible outerplanar blocks, and Corollary~\ref{cor:k23-threshold} reduces the root-reversible series-parallel case to blocks containing a $K_{2,3}$ minor.  The conjecture below is consistent with both boundaries established here.  It holds for $2$-connected outerplanar graphs by Theorem~\ref{thm:outerplanar}, whereas every non-seminorming family $N_{t,q}$ constructed in Section~4 contains a $K_4$ minor by Proposition~\ref{prop:cyclic-k4-minor} and therefore lies outside the series-parallel class.  Thus any counterexample would require a new mechanism inside the genuinely non-outerplanar series-parallel regime.

\begin{conjecture}[Series-parallel boundary]\label{conj:series-parallel}
Every $2$-connected series-parallel strongly dominating graph is seminorming.
\end{conjecture}

\begin{problem}[Series-parallel source classification]\label{prob:series-parallel-sources}
Classify the connected root-reversible series-parallel universal even-Schatten blocks with non-adjacent roots.  By Corollary~\ref{cor:k23-threshold}, it is enough to consider blocks containing a $K_{2,3}$ minor.  Corollary~\ref{cor:complete-bipartite-source-parity} settles the full two-rooted $K_{2,m}$ subfamily.
\end{problem}

\begin{problem}[Combinatorial classification of universal blocks]\label{prob:universal-schatten-blocks}
Classify all finite root-reversible birooted graphs satisfying the equivalent conditions of Theorem~\ref{thm:exact-cyclic-source}.  In particular, determine whether every such block admits a finite half-operator factorization certificate of the type used for $\mathcal B_{2t}$.
\end{problem}

A broader analytic problem is to determine whether strong domination is generated by a short list of closure principles: decorated H\"older for seminorming graphs, spectral moments for even paths, scalar H\"older for the one-root construction in Appendix~\ref{app:one-root}, and noncommutative H\"older for cyclic two-root amalgamation.

\clearpage
\appendix
\section{The scalar one-root comparison}\label{app:one-root}

A \emph{rooted graph} is a pair $(B,r)$ consisting of a graph $B$ and a distinguished vertex $r$.  For a rooted subgraph $S\subseteq B$, we retain the root as an isolated vertex if it is not incident with an edge of $S$.  Given a kernel $W$, let
\[
        f_S^W(x)
\]
denote the homomorphism density of $S$ with the root fixed at $x$, integrating over all other vertices; put $f_\emptyset^W\equiv1$.  For an integer $q\ge2$, write $B^{\vee q}$ for the graph obtained from $q$ disjoint copies of $B$ by identifying their roots and making no other identifications.  Rooted and partially labelled homomorphism densities are standard objects in graph algebras and graph-limit theory \cite{LovaszSzegedyContractors,LovaszSzegedyPositivstellensatz,LovaszBook}.

The one-root construction is the scalar counterpart of cyclic amalgamation.

\begin{theorem}[Even-fold rooted-amalgamation criterion]\label{thm:evenfold-rooted-certificate}
Let $(B,r)$ be a rooted graph with $b=e(B)\ge1$, and let $q\ge2$ be even.  Then the following are equivalent.
\begin{enumerate}[label=\textup{(\roman*)}]
\item The $q$-fold rooted amalgam $B^{\vee q}$ is strongly dominating.
\item For every bounded symmetric real kernel $W$ and every non-empty rooted subgraph $S\subseteq B$,
\begin{equation}\label{eq:general-evenfold-certificate}
        \|f_S^W\|_{L^q}
        \le
        \|f_B^W\|_{L^q}^{e(S)/b}.
\end{equation}
\end{enumerate}
In particular,
\[
        p_{B^{\vee q}}(W)=\|f_B^W\|_{L^q}^{1/b}.
\]
\end{theorem}

\begin{proof}
Every subgraph $F\subseteq B^{\vee q}$ decomposes, after isolated vertices are ignored, as a $q$-tuple of rooted subgraphs $S_1,\ldots,S_q\subseteq B$.  With the common root fixed at $x$,
\[
        t_F(W)=\int_0^1\prod_{i=1}^q f_{S_i}^W(x)\,dx.
\]
For the full amalgam, the copies are identical and $q$ is even, so
\[
        t_{B^{\vee q}}(W)
        =\int_0^1 f_B^W(x)^q\,dx
        =\|f_B^W\|_q^q.
\]
Since $e(B^{\vee q})=qb$, this gives the displayed formula for $p_{B^{\vee q}}$.

Assume first that $B^{\vee q}$ is strongly dominating.  Place the same rooted subgraph $S$ in all $q$ copies.  The resulting subgraph $S^{\vee q}$ satisfies
\[
        t_{S^{\vee q}}(W)=\|f_S^W\|_q^q,
        \qquad
        e(S^{\vee q})=q e(S).
\]
For non-empty $S$, the inequality $p_{B^{\vee q}}(W)\ge p_{S^{\vee q}}(W)$ is exactly
\eqref{eq:general-evenfold-certificate}.

Conversely, assume \eqref{eq:general-evenfold-certificate}.  Empty restrictions contribute the constant function $1$ and are omitted from the norm product and from the edge-count sum.  H\"older's inequality gives
\[
\begin{aligned}
        |t_F(W)|
        &\le \prod_{i=1}^q\|f_{S_i}^W\|_q \\
        &\le \|f_B^W\|_q^{\sum_i e(S_i)/b}
         =p_{B^{\vee q}}(W)^{e(F)}.
\end{aligned}
\]
Thus $p_F(W)\le p_{B^{\vee q}}(W)$ for every non-empty subgraph $F$.
\end{proof}

\begin{remark}[The analytic inequality behind amalgamation]\label{rem:holder-vs-bl}
The composition step in Theorem \ref{thm:evenfold-rooted-certificate} is exactly the generalized H\"older inequality on the common-root variable.  It is therefore unnecessary, and potentially misleading, to repackage the one-root criterion as a Brascamp--Lieb inequality.  The graphical Brascamp--Lieb inequality of Sah--Sawhney--Stoner--Zhao concerns non-negative edge-decorated integrals on triangle-free graphs and bounds them by complete-bipartite graph norms \cite{SahSawhneyStonerZhao}; it does not supply the signed conditional estimates required here.  For a two-vertex boundary, rooted functions become transfer operators and generalized H\"older is replaced by the Schatten H\"older inequality developed in Sections~3 and~4.
\end{remark}

\begin{proposition}[Why one-root iteration eventually stops]\label{prop:one-root-no-go}
Let $(B,r)$ be a rooted graph with $b=e(B)\ge1$ and $d_B(r)>0$.  Define
\[
\begin{aligned}
        \tau_0(B,r)&=\min\{|C|:C\text{ is a vertex cover of }B\text{ and }r\notin C\},\\
        \tau_1(B,r)&=\min\{|C|:C\text{ is a vertex cover of }B\text{ and }r\in C\}.
\end{aligned}
\]
Then, for every $q\ge1$,
\begin{equation}\label{eq:wedge-cover-formula}
        \tau(B^{\vee q})
        =\min\bigl\{q\tau_0(B,r),\ 1+q(\tau_1(B,r)-1)\bigr\}.
\end{equation}
If $B$ has an edge not incident with $r$, then $B^{\vee q}$ is not strongly dominating for all sufficiently large $q$.  Consequently, among rooted blocks with $d_B(r)>0$, a fixed block can produce strongly dominating one-root amalgams for infinitely many multiplicities only when, after isolated vertices are removed, it is a star centred at the root.
\end{proposition}

\begin{proof}
A vertex cover of $B^{\vee q}$ either contains the common root or does not.  In the second case its restriction to every copy is a cover excluding $r$, and the copies are otherwise disjoint, giving the first term in \eqref{eq:wedge-cover-formula}.  In the first case the common root is counted once and each copy requires $\tau_1(B,r)-1$ additional vertices, giving the second term.

Suppose that $B$ has an edge not incident with $r$.  Then $\tau_0(B,r)\ge1$ and $\tau_1(B,r)-1\ge1$, so \eqref{eq:wedge-cover-formula} grows linearly with $q$.  On the other hand, if
\[
        \Delta_* = \max\{d_B(v):v\ne r\},
\]
then
\[
        e(B^{\vee q})=qb,
        \qquad
        \Delta(B^{\vee q})=\max\{q d_B(r),\Delta_*\}.
\]
For all sufficiently large $q$ this gives
\[
        \frac{e(B^{\vee q})}{\Delta(B^{\vee q})}
        =\frac{b}{d_B(r)},
\]
a bounded quantity.  Hence eventually
\[
        \tau(B^{\vee q})>\frac{e(B^{\vee q})}{\Delta(B^{\vee q})},
\]
and Theorem \ref{thm:vertexcover} excludes the amalgam.

If every edge of $B$ is incident with $r$, then, after isolated vertices are removed, $B$ is a star centred at $r$, and $B^{\vee q}$ is again a star.  Thus, under the standing assumption $d_B(r)>0$, stars are the only fixed one-root blocks not ruled out at large multiplicity; the even-degree members are strongly dominating by Proposition \ref{prop:evenstars}.
\end{proof}

\begin{remark}
The condition $d_B(r)>0$ excludes a degenerate case in which the chosen root does not participate in the amalgamation.  For example, if $B=K_2\sqcup\{r\}$ with $r$ isolated, then $B^{\vee q}=qK_2\sqcup\{r\}$, and the amalgams are strongly dominating for every $q$.  For connected rooted blocks with at least one edge, the root-incidence condition is automatic.
\end{remark}

\end{document}